\documentclass[11pt]{amsart}

\usepackage[usenames]{xcolor}
\usepackage{graphicx}
\usepackage{hyperref}
\usepackage{subfig}
\hypersetup{
    linktocpage=true,
    colorlinks=true,
    linkcolor = blue,
    citecolor = purple,
}

\usepackage{amsthm}
\usepackage[centering,margin=1in]{geometry}

\def \integers{\mathbb{Z}}
\def \reals{\mathbb{R}}

\usepackage[numbers]{natbib}
\usepackage{amsmath,amssymb,enumerate}
\usepackage[all]{xy}
\usepackage{mathtools}
\usepackage{ytableau}
\ytableausetup{boxsize=2.4ex,centertableaux}

\theoremstyle{plain} 
\newtheorem{theorem}{Theorem}[section]
\newtheorem{lemma}[theorem]{Lemma}
\newtheorem{proposition}[theorem]{Proposition}
\newtheorem{corollary}[theorem]{Corollary}
\newtheorem{conjecture}[theorem]{Conjecture}

\theoremstyle{definition}
\newtheorem{example}[theorem]{Example}
\newtheorem{definition}[theorem]{Definition}

\theoremstyle{remark}
\newtheorem{remark}[theorem]{Remark}
\newtheorem{remarks}[theorem]{Remarks}

\newlength{\cellsize}
\newcommand\tableau[1]{
\vcenter{
\let\\=\cr
\baselineskip=-16000pt
\lineskiplimit=16000pt
\lineskip=0pt
\halign{&\tableaucell{##}\cr#1\crcr}}}

\newcommand{\tableaucell}[1]{{%
\def \arg{#1}\def \void{}%

\ifx \void \arg
\vbox to \cellsize{\vfil \hrule width \cellsize height 0pt}%
\else
\unitlength=\cellsize
\begin{picture}(1,1)
\put(0,0){\makebox(1,1){$#1$}}
\put(0,0){\line(1,0){1}}
\put(0,1){\line(1,0){1}}
\put(0,0){\line(0,1){1}}
\put(1,0){\line(0,1){1}}
\end{picture}%
\fi}}
\newcommand{\ct}{\mathrm{ct}}

\newcommand{\inner}[2]{
\langle #1, #2 \rangle 
}

\renewcommand{\l}[1] {
l^{J}_{#1}
}

\newcommand{\qstep}{
\lhd
}

\def \A{\mathcal{A}}
\def \B{B^{\otimes \lambda} }
\def \sgn{\mathrm{sgn}}
\def \sign{\mathrm{sign}}
\def \R{\mathbb{R}}
\def \Z{\mathbb{Z}}
\def \hh{\mathfrak{h}}
\def \g{\mathfrak{g}}
\def \wt{\mathrm{wt}}
\def \charge{ \mathrm{charge} }
\def \height{\mathrm{height} }
\def \htroot{\mathrm{ht}}
\def \fill{\mathrm{fill}}
\def \sfill{\mathrm{sfill}}
\def \Bzero{\mathbf{0}}

\begin{document}

\title[Quantum alcove model]{A generalization of the alcove model and its applications} 
\author{Cristian Lenart} 
\address{Department of Mathematics and Statistics, State University of New York at Albany,
Albany, NY 12222, USA}
\email{clenart@albany.edu}
\author{Arthur Lubovsky}
\address{Department of Mathematics and Statistics, State University of New York at Albany,
Albany, NY 12222, USA}
\email{alubovsky@albany.edu}
\thanks{Both authors were partially supported by the NSF grant DMS--1101264. The first author gratefully acknowledges the hospitality of the Max-Planck-Institut f\"ur Mathematik in Bonn, where part of this work was carried out.}

\begin{abstract}

    The alcove model of the first author and A. Postnikov uniformly 
    describes highest weight crystals of semisimple Lie algebras.  
    We construct a generalization, called the quantum alcove model. In joint work of the first author with S. Naito, D. Sagaki, A. Schilling, and M. Shimozono, this was shown to uniformly describe tensor products of column shape Kirillov-Reshetikhin crystals in all untwisted affine types; moreover, an efficient formula for the corresponding energy function is available. In
    the second part of this paper, we specialize the quantum alcove model to types $A$ and $C$. We give explicit affine crystal isomorphisms from the specialized quantum
    alcove model to the corresponding tensor products of column shape
    Kirillov-Reshetikhin crystals, which are realized in terms of Kashiwara-Nakashima columns. 
   

\end{abstract}
\maketitle

\section{Introduction}

Kashiwara's {\em crystals} \cite{kascbq} are colored directed graphs encoding the structure of certain bases (called crystal bases) for certain representations of quantum groups $U_q({\mathfrak g})$ as $q$ goes to zero. The first author and A. Postnikov \cite{lapawg,lapcmc} defined the so-called {\em alcove model} for highest weight crystals associated to a semisimple Lie algebra $\mathfrak g$ (in fact, the model was defined more generally, for symmetrizable Kac-Moody algebras $\mathfrak g$). A related model is the one of Gaussent-Littelmann, based on {\em LS-galleries} \cite{gallsg}. Both models are discrete counterparts of the celebrated {\em Littelmann path model} \cite{litlrr,litpro}.

In this paper we construct a generalization of the alcove model, which we
call the \emph{quantum alcove model}, as it is based on enumerating paths in
the so-called {\em quantum Bruhat graph} of the corresponding finite Weyl
group. This graph originates in the quantum
cohomology theory for flag varieties \cite{fawqps}, and was first studied in \cite{bfpmbo}. The path enumeration is
determined by the choice of a certain sequence of alcoves (an alcove
path or, equivalently, a $\lambda$-chain of roots), like in the classical alcove model. If we restrict to paths in
the Hasse diagram of the Bruhat order, we recover the classical alcove model. The
mentioned paths in the quantum Bruhat graph first appeared in
\cite{Lenart}, where they index the terms in the specialization $t=0$
of the Ram-Yip formula \cite{raycfm} for {\em Macdonald polynomials} $P_{\lambda}(X;q,t)$.
We construct combinatorial crystal operators on the mentioned paths, and prove various properties of them. 


The main application \citep{unialcmod,unialcmod2}
is that the new model uniformly describes tensor
products of column shape {\em Kirillov-Reshetikhin (KR) crystals}
\citep{karrym}, for all untwisted affine types. (KR crystals correspond to certain finite-dimensional  representations of affine algebras.) More precisely, the model realizes the crystal operators on the mentioned tensor product, and also gives an efficient formula (based on the so-called height statistic) for the corresponding {\em energy function} \cite{hkorff}. (The energy  can be viewed as an affine grading on a tensor product of KR crystals \cite{NS08,KRcrystals_energy}.) This result, combined with the Ram-Yip formula for Macdonald polynomials \cite{raycfm}, implies that the graded character of a tensor product of column shape KR modules (the grading being by the energy function) concides with the corresponding Macdonald polynomial specialized at $t=0$ \cite{unialcmod2}. 

In the second part of this paper, we specialize the quantum alcove model to types $A$ and $C$, and prove that the bijections
constructed in \citep{Lenart}, from the objects of the specialized quantum alcove
model to the tensor products of the corresponding {\em Kashiwara-Nakashima (KN) columns}
\citep{kancgr}, are affine crystal isomorphisms. (A column shape KR crystal is realized by a KN column in these cases.) Note that this result has no overlap with the type-independent result in \citep{unialcmod,unialcmod2}, because the $\lambda$-chains on which the quantum alcove model is based in the two cases are different. Moreover, note that having such explicit bijections to  models based on diagram fillings, which are known to be crystal isomorphisms, is important for the following reason. It is often easier to define extra structure on the quantum alcove model, which describes additional structure of the crystal, and only then translate it to the models based on fillings via the bijection. One such example was mentioned above, in connection with the energy function; the height statistic in the quantum alcove model has been translated to the so-called {\em charge} statistic for fillings in types $A$ and $C$ in \cite{Lenart}, while type $B$ is currently under investigation in \cite{balcst}. Another example is related to the combinatorial $R$-matrix mentioned below. 

We also conjecture that, like the alcove model, its generalization given here is independent of the $\lambda$-chain of roots on which the whole construction is based, cf. \cite{lenccg}. This conjecture is currently under investigation in \cite{lalurc}. We intend to realize an affine crystal isomorphism between the models based on two $\lambda$-chains by extending to the quantum
alcove model the alcove model version of Sch\"utzenberger's {\em jeu
    de taquin} \cite{fulyt} on Young tableaux in \cite{lenccg}; the latter is based
on so-called {\em Yang-Baxter moves}.
Another application of this construction would be a uniform realization of the  {\em
    combinatorial $R$-matrix} (i.e., the unique affine crystal isomorphism
commuting factors in a tensor product of KR crystals).

\section{Background}

Let $\mathfrak{g}$ be a complex semisimple Lie algebra, and $\mathfrak{h}$ a Cartan subalgebra, whose rank is $r$. 
Let $\Phi \subset \mathfrak{h}^*$ be the corresponding irreducible \emph{root system}, 
$\mathfrak{h}^*_{\reals}\subset \mathfrak{h}$ the real span of the roots, and $\Phi^{+} \subset \Phi$ the set of positive roots.
Let $\Phi^{-} := \Phi \backslash \Phi^{+}$.
For $ \alpha \in \Phi$ we will say that $ \alpha > 0 $ if $ \alpha \in \Phi^{+}$,
and $ \alpha < 0 $ if $ \alpha \in \Phi^{-}$.
The sign of the root $\alpha$, denoted $\sgn(\alpha)$, is defined to be $1$ if $\alpha \in \Phi^{+}$, and $-1$ otherwise.
Let $| \alpha | = \sgn( \alpha ) \alpha $.
Let $\rho := \frac{1}{2}(\sum_{\alpha \in \Phi^{+}}\alpha)$. Let 
$\alpha_1, \ldots , \alpha_r \in \Phi^{+}$ be the corresponding \emph{simple roots}, and $s_i:=s_{\alpha_i}$ the corresponding simple reflections. We denote 
$\inner{\cdot}{\cdot}$ the nondegenerate scalar product on 
$\mathfrak{h}^{*}_{\reals}$ induced by the Killing form. Given a root $\alpha$, we consider the corresponding \emph{coroot} $\alpha^{\vee} := 2\alpha/ \inner{\alpha}{\alpha}$ and reflection
$s_{\alpha}$.
If $\alpha= \sum_i c_i \alpha_i$, then the \emph{height} of $\alpha$,
denoted by $\htroot(\alpha)$, is given by 
$\htroot(\alpha):=\sum_i c_i$. 
We denote by $\widetilde{\alpha}$ the highest root in $\Phi^{+}$; we let 
$\theta = \alpha_0 := -\widetilde{\alpha} $ and $s_0:=s_{\widetilde{\alpha}}$.

Let $W$ be the corresponding \emph{Weyl group}. The length function on $W$ is denoted by $\ell(\cdot)$. The \emph{Bruhat order} on $W$ is defined by its covers $w \lessdot ws_{\alpha}$, for $\alpha \in \Phi^{+}$,
if $\ell(ws_{\alpha}) = \ell(w) + 1$. 
Define 
$w \qstep ws_{\alpha}$, for $\alpha \in \Phi^{+}$, if 
$\ell(ws_{\alpha}) =  \ell(w) - 2\htroot( \alpha^\vee )  + 1$.
The \emph{quantum Bruhat graph} \cite{fawqps} is the directed graph on $W$ with edges labeled by positive roots
\begin{equation}
	w \stackrel{\alpha}{\longrightarrow} 
	ws_{\alpha} \quad 
	\text{ for  } 
	w \lessdot ws_{\alpha} \,\mbox{ or }\, w \qstep ws_{\alpha}\,;
	\label{eqn:qbruhat_edge}
\end{equation}
see Example \ref{qbgex}. 

The \emph{weight lattice} $\Lambda$ is given by 
\begin{equation}
	\Lambda := \left\{ \lambda \in \mathfrak{h}_{\reals}^{*} \, : \,
	\inner{\lambda}{\alpha^{\vee}} \in \integers \text{ for any } \alpha \in \Phi \right\}.
	\label{eqn:weight_lattice}
\end{equation}
The weight lattice $\Lambda$ is generated by the \emph{fundamental weights} 
$\omega_1, \ldots \omega_r$, which form the dual basis to the basis of simple coroots, i.e., 
$\inner{\omega_i}{\alpha_j^{\vee}}= \delta_{ij}$. The set $\Lambda^{+}$ of \emph{dominant weights}
is given by 
\begin{equation}
	\Lambda^{+} := \left\{  \lambda \in \Lambda \, : \, 
	\inner{\lambda}{\alpha^{\vee}} \geq 0 \text{ for any } \alpha \in \Phi^{+}
	\right\}.
	\label{eqn:dominant_weights}
\end{equation}

Given $\alpha \in \Phi$ and $k \in \integers$, we denote by $s_{\alpha,k}$ the reflection in the affine hyperplane
\begin{equation}
	H_{\alpha,k}:= \left\{  \lambda \in \mathfrak{h}^{*}_{\reals} \, : \, 
	\inner{\lambda}{\alpha^{\vee}} = k \right\}
	\label{eqn:affine_hyperplane}.
\end{equation}
These reflections generate the \emph{affine Weyl group} $W_{\textrm{aff}}$ for the 
\emph{dual root system} $\Phi^{\vee}:= \left\{ \alpha^{\vee} \, |\, \alpha \in \Phi \right\}$.
The hyperplanes $H_{\alpha,k}$ divide the real vector space $\mathfrak{h}^{*}_\reals$ into open regions, called \emph{alcoves.} The \emph{fundamental alcove} $A_{\circ}$ is given by
\begin{equation}
	A_{\circ} := \left\{ \lambda \in \mathfrak{h}_{\reals}^{*} \, | \, 
	0 < \inner{\lambda}{\alpha^{\vee}} < 1 \text{ for all } \alpha \in \Phi^{+}
	\right\}.
	\label{eqn:fundamental_alcove}
\end{equation}

We will need the following properties of the quantum Bruhat graph which were proved in \cite{unialcmod}; more precisely, Lemma \ref{lemma:theta} below is a simplified version of Proposition 5.4.2 in the mentioned paper, whereas Lemmas \ref{prop:deodhar} and \ref{prop:deodhar0} below are simplified versions of different parts of the Diamond Lemma 5.5.2.
\begin{lemma}
	Let $w \in W$. We have $w^{-1}(\theta)>0$ if and only if $w \qstep s_{\theta}w$. 
	We also have $w^{-1}(\theta) < 0$ if and only if $s_{\theta}w \qstep w$.
	\label{lemma:theta}
\end{lemma}

\begin{lemma}
	\label{prop:deodhar}
	Let $w \in W$, let $\alpha$ be a simple root, $\beta \in \Phi^+$, and assume $s_{\alpha}w \ne ws_{\beta}$. 
	Then $w \lessdot s_{\alpha}w$ and $ w \longrightarrow ws_{\beta}$ if and only if
	  $ws_{\beta} \lessdot s_{\alpha}ws_{\beta}$ and 
	$s_{\alpha}w \longrightarrow s_{\alpha}ws_{\beta}$, cf. the diagram below. 
	Furthermore, in this context we have $ w \lessdot ws_{\beta}$  if and only if
	$s_{\alpha}w \lessdot s_{\alpha}ws_{\beta}$. 
        	\begin{displaymath}
		\xymatrix{
		& s_{\alpha} ws_{\beta} & \\
		s_{\alpha}w \ar[ru] & & ws_{\beta} \ar[lu] \\
		& w \ar[ru] \ar[lu] &  }
	\end{displaymath}

\end{lemma}
\begin{lemma}
	\label{prop:deodhar0}
	Let $w \in W$,  
	$\beta \in \Phi^+$, 
	and assume $s_{\theta}w \ne ws_{\beta}$. 
	Then $w \qstep s_{\theta}w$ and
	$ w \longrightarrow ws_{\beta}$ if and only if
	  $ws_{\beta} \qstep s_{\theta}ws_{\beta}$ 
	  and 
	$s_{\theta}w \longrightarrow s_{\theta}ws_{\beta}.$ 
\end{lemma}
\subsection{Kirillov-Reshetikhin (KR) crystals}
\label{subsection:KR-crystals}

A $\g$-{\em crystal} (for a symmetrizable Kac-Moody $\g$) is a nonempty set $B$ together with maps $e_i,f_i:B\to B\cup 
\{ \Bzero \}$ for $i\in I$ ($I$ indexes the simple roots, as usual, and $\Bzero \not \in B$), and 
$\wt:B \to \Lambda$. We require $b'=f_i(b)$ if and only if $b=e_i(b')$, and $\wt(f_i(b))=\wt(b)-\alpha_i$. The maps $e_i$ and 
$f_i$ are called {\em crystal operators} 
and are represented as arrows $b \to b'=f_i(b)$ colored $i$; thus they endow $B$ with the structure of a colored directed graph.
For $b\in B$, we set $\varepsilon_i(b) := \max\{k \mid e_i^k(b) \neq \Bzero \}$, and 
$\varphi_i(b) := \max\{k \mid f_i^k(b) \neq \Bzero \}$.
Given two $\g$-crystals $B_1$ and $B_2$, we define their tensor product $B_1 \otimes B_2$ 
as follows.
As a set, $B_1\otimes B_2$ is the Cartesian
product of the two sets. For $b=b_1 \otimes b_2\in B_1 \otimes B_2$, the weight function is simply
$\wt(b) := \wt(b_1) + \wt(b_2)$. 
The crystal operators $f_i$ and $e_i$ are given by
\begin{equation}\label{tens1}
	f_i (b_1 \otimes b_2):=
	\begin{cases}
	f_i  (b_1) \otimes b_2 & \text{if $\varepsilon_i(b_1) \geq \varphi_i(b_2)$}\\
    b_1 \otimes f_i (b_2)  & \text{ otherwise }
	\end{cases}
\end{equation}
\begin{equation}\label{tens2}
	e_i (b_1 \otimes b_2):=
	\begin{cases}
	e_i  (b_1) \otimes b_2 & \text{if $\varepsilon_i(b_1) > \varphi_i(b_2)$}\\
    b_1 \otimes e_i (b_2)  & \text{ otherwise .}
	\end{cases}
\end{equation}
The {\em highest weight crystal} $B(\lambda)$ of highest weight $\lambda\in \Lambda^+$ is a 
certain crystal with a unique element 
$u_\lambda$ such that $e_i(u_\lambda)=\Bzero$ for all $i\in I$ 
and $\wt(u_\lambda)=\lambda$.
It encodes the structure of the crystal basis of the $U_q(\mathfrak{g})$-irreducible 
representation with highest weight $\lambda$ as $q$ goes to 0.

A {\em Kirillov-Reshetikhin (KR) crystal} \cite{karrym} is a finite crystal $B^{r,s}$ for an affine algebra, associated to a rectangle of height $r$ and width $s$, where $r\in I\setminus\{0\}$ and $s$ is any positive integer. We refer, throughout, to the untwisted affine types $A_{n-1}^{(1)}-G_2^{(1)}$. 

We now describe the models based on diagram fillings for KR crystals $B^{r,1}$ of type $A_{n-1}^{(1)}$ 
and $C_n^{(1)}$, 
where
$r\in \{1,2,\ldots,n-1\}$ and $r\in \{1,2,\ldots,n\}$, respectively. 
As a classical type $A_{n-1}$ 
(resp. $C_n$) 
crystal, 
the KR crystal $B^{r,1}$ is isomorphic to the corresponding 
$B(\omega_r)$. Therefore, we can use the corresponding models in terms of fillings, as mentioned below.

In type $A_{n-1}^{(1)}$, an element $b\in B^{r,1}$ is 
represented by a strictly increasing filling 
of a height $r$ column, with entries in $[n]:=\{1, \dots , n \}$. 
We will now describe the crystal operators on a tensor product of type $A_{n-1}^{(1)}$ KR crystals $B^{r,1}$ in terms of the so-called 
signature rule, which is just a translation of the tensor product rules (\ref{tens1})-(\ref{tens2}). 
To apply $f_i$ (or $e_i$) on $b := b_1 \otimes \cdots \otimes b_k$ in 
$B^{i_1,1} \otimes \cdots \otimes B^{i_k,1}$, consider the word with letters
$i$ and $i+1$, if $1\leq i \leq n-1$
(resp., the letters $n$ and $1$, if $i=0$) formed by recording these letters in $b_1,\ldots,b_k$, which are scanned from left to right and bottom to top; we make the convention that if $i=0$ and a column contains both $1$ and $n$, then we discard this column. We replace the letter $i$ with the symbol $+$ and the letter $i+1$ with $-$ (resp., $n$ with $+$ and $1$ with $-$, if $i=0$). 
Then, we remove from our binary word adjacent pairs $-+$, as long as this is possible. At the end of this process, we are left with a word
\begin{equation}
    \rho_i(b) = \underbrace{++ \ldots +}_x\underbrace{-- \ldots -}_y\,,
    \label{eqn:reduced_word}
\end{equation}
called the $i$-signature of $b$.
\begin{definition}
{\rm (1)} If $y > 0$, then $e_i(b)$ is obtained by replacing  in $b$ the letter 
        $i+1$ which corresponds to the leftmost $-$ in
        $\rho_i(b)$ with the letter $i$ (resp., the letter $1$ with $n$, after which we sort the column, if $i=0$). 
        If $y=0$, then $e_i(b) = \Bzero$.

 {\rm (2)} If $x > 0$, then $f_i(b)$ is obtained by replacing  in $b$ the
        letter $i$ which corresponds to the rightmost $+$ in
        $\rho_i(b)$ with the letter $i+1$ (resp., 
        the letter $n$ with $1$, after which we sort the column, if $i=0$).
        If $x=0$, then $f_i(b) =\Bzero$. 
\label{definition:crystal_operators}
\end{definition}

\begin{example}
	\label{example:root0}
	Let $n=3$, 	$b = \tableau{2 & 1 & 1 \\ 3 & 2 }$ $\mapsto$
    $\ytableaushort{2,3} \otimes \ytableaushort{1,2} \otimes
    \ytableaushort{1}$, and has $+--$ as its $0$-signature. So we have $f_0\left(\ytableaushort{211,32}\right) =
    \ytableaushort{111,22}$.
\end{example}

In type $C_{n}^{(1)}$, the elements of $B^{r,1}$ are represented by {\em Kashiwara-Nakashima (KN) columns} 
\citep{kancgr} of height $r$, with
entries in the set $[\overline{n}]:= 
\left\{ 1< \dots< n<\overline{n}<  \dots <\overline{1} \right\}$, which
we will now describe.
\begin{definition}
   A column-strict filling $C = x_1 \ldots x_r$  with entries in
   $[\overline{n}]$ is a KN column if there is no pair $(z,\overline{z})$
       of letters in $C$ such that:
       \[
           z=x_p,\; \overline{z} = x_q,\; q-p \leq r-z.
       \]
       \label{definition:KN_columns}
\end{definition}

Crystal operators $f_i$ and $ e_i$ are defined on tensor products of KN
columns in a similar way to type $A_{n-1}^{(1)}$. 
To apply $f_i$ (or $e_i$) on $b := b_1 \otimes \cdots \otimes b_k$ in 
$B^{i_1,1} \otimes \cdots \otimes B^{i_k,1}$, consider the word with letters
$i, i+1, \overline{\imath}, \overline{i+1}$, if $1\leq i \leq n-1$
(resp., the letters $n$ and $\overline{n}$, if $i=n$, or  $\overline{1}$ and $1$, if $i=0$) formed by recording these letters in $b_1,\ldots,b_k$, which are scanned from from left to right and bottom to top; note that the letters $1$ and $\overline{1}$ cannot simultaneously appear in a column, so we do not need an exception like in type $A_{n-1}^{(1)}$ if $i=0$. We replace the letters $i,\overline{i+1}$ with the symbol $+$ 
and the letters 
$i+1, \overline{\imath}$ with $-$, if $1\leq i \leq n-1$ 
(if $i=n$ we replace $n$ with $+$ and $\overline{n}$ with $-$, and if $i=0$ we replace $\overline{1}$ with $+$ and $1$ with $-$). We proceed like in type $A_{n-1}^{(1)}$ by cancelling adjacent pairs $-+$ as long as possible, and we obtain the $i$-signature $\rho_i(b)$. 
The crystal operators $f_i$ and $e_i$ are again given in terms $\rho_i(b)$, by a similar procedure to the one in Definition $\ref{definition:crystal_operators}$.
Namely, if $1\leq i \leq n-1$, changing $+$ to $-$ means changing $i$ to $i+1$, if
$+$ corresponds to $i$, and changing
$\overline{i+1}$ to $\overline{\imath}$, if $+$ corresponds to
$\overline{i+1}$; similarly changing $-$ to $+$ means changing
$i+1$ to $i$ or $\overline{\imath}$ to $\overline{i+1}$. On another hand, changing $+$ to $-$ means changing $n$ to $\overline{n}$ if $i=n$, and changing $\overline{1}$ to $1$ if $i=0$. 

We will need a different definition of KN columns which was proved to
be equivalent to the one above in \cite{shsjdt}.

\begin{definition}
Let $C$ be column and $I = \left\{ z_1 > \ldots > z_k \right\}$ the
set of 
unbarred letters $z$ such that the pair $(z,\overline{z})$ occurs in
$C$. The column $C$ can be \emph{split} when there exists a set of $k$
unbarred letters $J = \left\{ t_1 > \ldots > t_k \right\} \subset [n]$
such that:
\begin{itemize}
    \item $t_1$ is the greatest letter in $[n]$ satisfying: 
        $t_1 < z_1$, $t_1\not \in C$, and $\overline{t_1} \not \in C$,
    \item  for $i = 2, \ldots, k$, the letter $t_i$ is the greatest
        one in $[n]$ satisfying $t_i < \min(t_{i-1}, z_i)$, 
        $t_i \not \in C$, and $\overline{t_i} \not \in C$.
\end{itemize}

In this case we write:

\begin{itemize}
    \item  $rC$ for the column obtained by changing $\overline{z_i}$
    into $\overline{t_i}$ in $C$ for each letter $z_i \in I$, and by
    reordering if necessary,
    \item $lC$ for the column obtained by changing $z_i$ into
        $t_i$ in $C$ for each letter $z_i \in I$, and by reordering if
        necessary.
\end{itemize}
The pair $(lC,rC)$ will be called a \emph{split column}, which we well
sometimes denote by $lCrC$.
\label{definition:KN_doubled_columns}
\end{definition}

\begin{example}
    The following is a KN column of height 5 in type $C_n$ for 
    $n \ge 5$, together with the corresponding split column:

    \[ 
        C=
        \tableau{4 \\ 5\\ \overline{ 5 }\\ \overline{ 4 }\\ \overline{ 3 } }\,,
        \quad 
        (lC,rC) =
        \tableau{1 \\ 2\\ \overline{ 5 }\\ \overline{ 4 }\\ \overline{ 3 } }
        \tableau{4 \\ 5\\ \overline{ 3 }\\ \overline{ 2 }\\ \overline{ 1 } }
    \]
We used the fact that $ \left\{ z_1 > z_2 \right\}= \left\{ 5>4
\right\} $, so $ \left\{ t_1 > t_2 \right\}= \left\{ 2 > 1 \right\} $.
\end{example}

\begin{proposition}\label{typeCcryst}
 ~
\begin{itemize}
\item[{\rm (1)}] {\rm \cite{shsjdt}} A column with entries in $[\overline{n}]$ is a KN column if and only if it can be split. 
\item[{\rm (2)}] {\rm \citep[Theorem 5.1]{kasscb}} The splitting is compatible with the action of the crystal operators, i.e., if $f_i(C)=C'$ then $lC'rC'=f_i^2(lCrC)$. This holds more generally, for tensor products of columns, i.e., if $f_i(C_1\ldots C_n)=C_1'\ldots C_n'$, then $lC_1'rC_1'\ldots lC_n'rC_n'=f_i^2(lC_1rC_1\ldots lC_nrC_n)$. 
\end{itemize}
\end{proposition}

In what follows we will use Definition \ref{definition:KN_doubled_columns} as the
definition of KN columns.

We refer again to KR crystals of arbitrary (untwisted) type. Let $\lambda=(\lambda_1  \geq \lambda_2\geq \ldots)$ be a partition, which encodes a dominant weight in classical types; let $\lambda'$ be the conjugate partition.
We define 
\begin{equation}\label{btensl}\B:=\bigotimes_{i=1}^{\lambda_1} B^{\lambda_i',1}\,,
\end{equation}
 assuming that the corresponding column shape KR crystals exist. We denote such a tensor product generically by $B$. It is known that $B$ is connected as an affine crystal, but disconnected as a classical crystal (i.e., with the $0$-arrows removed). 

\begin{definition}
An arrow $b\rightarrow f_i(b)$ in $B$ is called a \emph{Demazure arrow} if $i \ne 0$, or $i=0$ and 
$\varepsilon_0(b)\ge 1$.
\end{definition}
Demazure arrows exclude $0$-arrows at the beginning of a string of
$0$-arrows.
We are interested in excluding $0$-arrows at the end of a string of
$0$-arrows. We call these arrows dual Demazure.
\begin{definition}
	\label{definition:pdemazure_arrows}
An arrow $b\rightarrow f_i(b)$ in $B$ is called a \emph{dual Demazure arrow} if $i \ne 0$, or $i=0$ and 
$\varphi_0(b)\geq 2$. 
\end{definition}

The \emph{energy function} $D=D_B$ 
is a function from $B$ to the integers, defined 
by summing the so-called local 
energies of all pairs of tensor factors \cite{hkorff}. We will only refer here to the so-called tail energy \cite{unialcmod2}, so we will not make this specification. (There are two conventions in defining the local energy of a pair of tensor factors: commuting the right one towards the head of the tensor product, or the left one towards the tail; the tail energy corresponds to the second choice.) 
We will only need the following property of the energy function, which defines it as an affine grading on $B$. 

\begin{theorem}{\rm \cite{NS08,KRcrystals_energy}}
	\label{theorem:energy_recursion}
The energy is preserved by the classical crystal operators $f_i$, i.e., $i\ne 0$.
	If $b\rightarrow f_0(b)$ is a dual Demazure arrow, then $D(f_0(b))=D(b)-1$.
\end{theorem}

It follows that the energy is determined up to a constant on the  
connected components of the 
subgraph of the 
affine crystal $B$ containing only the dual Demazure arrows.
In the case when all of the tensor factors of $B$ are \emph{perfect}
crystals \citep{hkqgcb}, the mentioned subgraph is connected, so the energy is determined up to a constant on the whole crystal $B$.

\begin{remark}\label{allperf}
In classical types, $B^{k,1}$ is perfect as follows: in types $A^{(1)}_{n-1}$ and $D^{(1)}_{n}$ for all $k$, in type $B^{(1)}_{n}$ only for $k\ne n$, and in type $C^{(1)}_n$ only for $k=n$ (using the standard indexing of the Dynkin diagram); in other words, for all the Dynkin nodes in simply-laced types, and only for the nodes corresponding to the long roots in non-simply-laced types. It was conjectured in \cite{hkorff} that the same is true in the exceptional types. In type $G_2^{(1)}$ this was confirmed in \cite{yampfg}, while for types $E_{6,7}^{(1)}$ and $F_{4}^{(1)}$ it was checked by computer, based on a model closely related to the quantum alcove model, see \cite{unialcmod2}.
\end{remark}

One can define a statistic called {\em charge} on the model based on KN columns for $\B$ in types $A$ and $C$. This was done in \cite{Lenart}, by translating 
a certain statistic in the Ram-Yip formula for Macdonald polynomials
(i.e., the height statistic in \eqref{eqn:height_statistic})
to the model based on KN columns, via 
certain bijections recalled in Section \ref{modelac}, cf. Remarks \ref{nddA} (2) and \ref{nddC} (2). In type $A$, this procedure leads to the same statistic that was originally defined by Lascoux and Sch\"utzenberger \cite{lassuc}. A similar procedure to the one in \cite{Lenart} is under investigation in type $B$ in \cite{balcst}. 
The charge statistic is related to the energy function by the following theorem.

\begin{theorem}{\rm \cite{naycharge,energy_charge}}
	\label{theorem:charge}
Let $\B$ be a tensor product of KR crystals in type $A_{n-1}^{(1)}$ or type $C_{n}^{(1)}$. For all $b\in \B$, we have
	$D(b) = - \charge(b)\,.$
\end{theorem}

The charge gives a much 
easier method to compute the energy than the recursive one based on Theorem \ref{theorem:energy_recursion}. See also Theorem \ref{mainconj} (2) for a much more general result.

\section{The quantum alcove model}

In this section we construct the quantum alcove model and study its main properties. 

\subsection{\texorpdfstring{$\lambda$-chains and admissible subsets}{lambda-chains and
admissible subsets}}
\label{subsection:chains_and_admissible_sequences}
We say that two alcoves are \emph{adjacent} if they are distinct and have a common wall. Given a pair of adjacent alcoves $A$ and $B$, we write $A \stackrel{\beta}{\longrightarrow} B$ if the
common wall is of the form $H_{\beta,k}$ and the root
$\beta \in \Phi$ points in the direction from $A$ to $B$.

\begin{definition}{\rm \cite{lapawg}}
	An  \emph{alcove path} is a {sequence of alcoves} $(A_0, A_1, \ldots, A_m)$ such that
	$A_{j-1}$ and $A_j$ are adjacent, for $j=1,\ldots m.$ We say that an alcove path 
	is \emph{reduced} if it has minimal length among all alcove paths from $A_0$ to $A_m$.
	\end{definition}
	
	Let $A_{\lambda}=A_{\circ}+\lambda$ be the translation of the fundamental alcove $A_{\circ}$ by the weight $\lambda$.
	
	\begin{definition}{\rm \cite{lapawg}}
		The sequence of roots $(\beta_1, \beta_2, \dots, \beta_m)$ is called a
		\emph{$\lambda$-chain} if 
		\[	
		A_0=A_{\circ} \stackrel{-\beta_1}{\longrightarrow} A_1
		\stackrel{-\beta_2}{\longrightarrow}\dots 
		\stackrel{-\beta_m}{\longrightarrow} A_m=A_{-\lambda}\]
is a reduced alcove path.
\end{definition}

We now fix a dominant weight $\lambda$ and an alcove path $\Pi=(A_0, \dots , A_m)$ from 
$A_0 = A_{\circ}$ to $A_m = A_{-\lambda}$. Note that $\Pi$ is determined by the corresponding $\lambda$-chain $\Gamma:=(\beta_1, \dots, \beta_m)$, which consists of positive roots. 
A specific choice of a $\lambda$-chain, called a lex $\lambda$-chain and denoted $\Gamma_{\rm lex}$, is given in \cite{lapcmc}[Proposition 4.2]; this choice depends on a total order on the simple roots. 
We let $r_i:=s_{\beta_i}$, and let $\widehat{r_i}$ be the affine reflection in the hyperplane containing the common face of $A_{i-1}$ and $A_i$, for $i=1, \ldots, m$; in other words, 
$\widehat{r}_i:= s_{\beta_i,-l_i}$, where $l_i:=|\left\{ j<i \, ; \, \beta_j = \beta_i \right\} |$. 
We define $\widetilde{l}_i:= \inner{\lambda}{\beta_i^{\vee}} -l_i = |\left\{ j \geq i \, ; \, \beta_j = \beta_i \right\} |$.
\begin{example}
	\label{example:lambda_chain}
	Consider the dominant weight $\lambda=3\varepsilon_1+2\varepsilon_2$ 
	in the root system $A_2$ (cf. Section \ref{subsection:TypeA} 
	and the notation therein).
	A $\lambda$-chain is $\Gamma=(\alpha_{23},\alpha_{13},\alpha_{23},\alpha_{13},\alpha_{12},\alpha_{13})$. 
	The corresponding $l_i$ are $(0,0,1,1,0,2)$ and
	$\widetilde{l}_i$ are $\left\{ 2,3,1,2,1,1 \right\}$. 
	The alcove path is shown in Figure \ref{fig:unfolded_chain}; here 
	$A_\circ$  is shaded, and $A_{-\lambda}$ is the alcove 
	at the end of the path.
\end{example}
\begin{figure}[h]
    \centering
\subfloat[$\Gamma$ for $\lambda=3\varepsilon_1+2\varepsilon_2$ \label{fig:unfolded_chain}]
{\includegraphics[scale=.41]{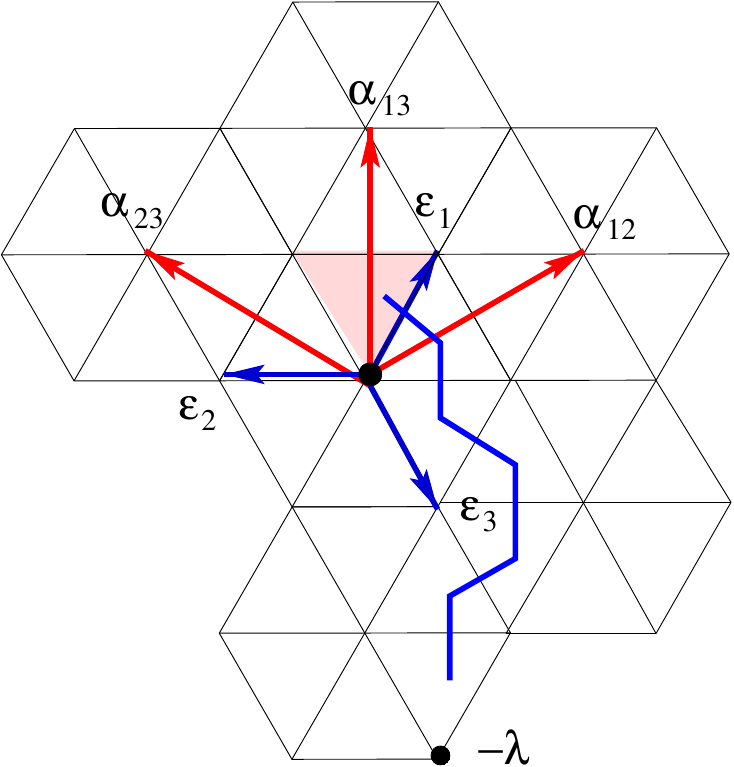}}
\hspace{10pt}
\subfloat[$\Gamma(J)$ for $J=\left\{ 1,2,3,5 \right\}$
		\label{fig:folded_chain}]
{\includegraphics[scale=.41]{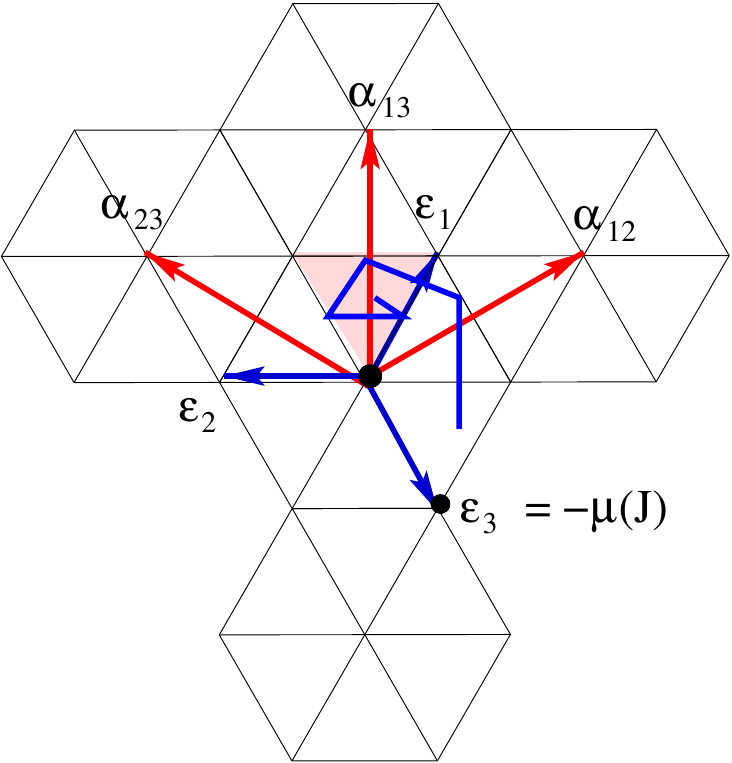}}
	\label{fig:gamma_and_delta}
	\caption{Unfolded and folded $\lambda$-chain}
\end{figure}

Let $J=\left\{ j_1 < j_2 < \cdots < j_s \right\}$  be a subset of $[m]$. The elements of $J$ are called \emph{folding positions}. We fold $\Pi$ in the hyperplanes corresponding to these positions and obtain a folded path, see Example \ref{example:folded_lambda_chain} and Figure \ref{fig:folded_chain}. Like $\Pi$,  the folded path can be recorded by a sequence of roots, namely 
$\Delta = \Gamma(J)=\left( \gamma_1,\gamma_2, \dots, \gamma_m \right)$; 
	 here 
	 \begin{equation}\label{defw}
	 \gamma_k:=r_{j_1}r_{j_2}\dots r_{j_p}(\beta_k)\,,
	 \end{equation}
	  with $j_p$ the 
	 largest folding position less than $k$. 
	 We define $\gamma_{\infty} := r_{j_1}r_{j_2}\dots r_{j_s}(\rho)$.
	 Upon folding, the hyperplane separating the alcoves $A_{k-1}$ and $A_k$ in $\Pi$ is mapped to 
\begin{equation}\label{deflev}
H_{|\gamma_k|,-\l{k}}=\widehat{r}_{j_1}\widehat{r}_{j_2}\dots \widehat{r}_{j_p}(H_{\beta_k,-l_k})\,,
\end{equation}
for some $\l{k}$, which is defined by this relation.

	 	 Given $i \in J$, we  say that $i$ is a \emph{positive folding position} if 
	 $\gamma_i>0$, and a \emph{negative folding position} if $\gamma_i<0$. 
	 We denote the positive folding positions by $J^{+}$, and the negative 
	 ones by $J^{-}$.
	 We call 
$\mu=\mu(J):=-\widehat{r}_{j_1}\widehat{r}_{j_2}\ldots \widehat{r}_{j_s}(-\lambda)$ the 
\emph{weight} of $J$.
We define 
\begin{equation}
	\label{eqn:height_statistic}
	\height(J):= \sum_{j \in J^{-}} \widetilde{l}_j.
\end{equation}
\begin{definition}
	A subset $J=\left\{ j_1 < j_2 < \cdots < j_s \right\} \subseteq [m]$ (possibly empty)
 is an \emph{admissible subset} if
we have the following path in the quantum Bruhat graph on $W$:
\begin{equation}
	\label{eqn:admissible}
	1 \stackrel{\beta_{j_1}}{\longrightarrow} r_{j_1} \stackrel{\beta_{j_2}}{\longrightarrow} r_{j_1}r_{j_2} 
	\stackrel{\beta_{j_3}}{\longrightarrow} \cdots \stackrel{\beta_{j_s}}{\longrightarrow} r_{j_1}r_{j_2}\cdots r_{j_s}\,.
\end{equation}
We call $\Delta=\Gamma(J)$ an \emph{admissible folding}.
We let $\A(\Gamma)$ be the collection of 
admissible subsets.
\end{definition}

\begin{remark}\label{spec}
If we restrict to admissible subsets for which the path \eqref{eqn:admissible} has no down steps, we recover the classical alcove model in \cite{lapawg,lapcmc}.
\end{remark}

\begin{example}
	\label{example:folded_lambda_chain}
	We continue Example \ref{example:lambda_chain}.
	Let $J =  \left\{ 1,2,3,5 \right\}$, then 
	$\Delta=\Gamma(J)=(\alpha_{23},\alpha_{12}, \alpha_{31}, \alpha_{23},$ \linebreak $\alpha_{21},\alpha_{13})$. The folded path is shown in Figure \ref{fig:folded_chain}.
We have $J^{+}=\left\{ 1,2\right\}$, $J^{-}=\left\{ 3,5 \right\}$, 
$\mu(J)=-\varepsilon_3$, and $\height(J)=2$.
In Section \ref{subsection:TypeA} 
we will describe an easy way to verify that $J$ is admissible. 
\end{example}

\subsection{Crystal operators}
\label{subsection:crystalop}
In this section we define the crystal operators $f_i$ and $e_i$. 
Given $J\subseteq [m]$ and
$\alpha\in \Phi$, we will use the following notation:
	\[ 
	 I_\alpha = I_{\alpha}(\Delta):= \left\{ i \in [m] \, | \, \gamma_i = \pm \alpha \right\}\,, \qquad 
\widehat{I}_\alpha = \widehat{I}_{\alpha}(\Delta):= I_{\alpha} \cup \{\infty\}\,, 
\]
and $l_{\alpha}^{\infty}:=\inner{\mu(J)}{\sgn(\alpha)\alpha^{\vee}}$.
The following graphical representation of the heights $l_i^J$ for $i\in{I}_\alpha$ and $l_{\alpha}^{\infty}$ 
is useful for defining the crystal operators.
Let 
\[\widehat{I}_{\alpha}= \left\{ i_1 < i_2 < \dots < i_n \leq m<i_{n+1}=\infty \right\}\,
	\text{ and  }
	\varepsilon_i := 
	\begin{cases}
		\,\,\,\, 1 &\text{ if } i \not \in J\\
		-1 & \text { if } i \in J
	\end{cases}.\,
	\]
If $\alpha > 0$, we define the continuous piecewise linear function 
$g_{\alpha}:[0,n+\frac{1}{2}] \to \reals$ by
\begin{equation}
	\label{eqn:piecewise-linear_graph}
	g_\alpha(0)= -\frac{1}{2}, \;\;\; g'_{\alpha}(x)=
	\begin{cases}
		\sgn(\gamma_{i_k}) & \text{ if } x \in (k-1,k-\frac{1}{2}),\, k = 1, \ldots, n\\
		\varepsilon_{i_k}\sgn(\gamma_{i_k}) & 
		\text{ if } x \in (k-\frac{1}{2},k),\, k=1,\ldots,n \\
		\sgn(\inner{\gamma_{\infty}}{\alpha^{\vee}}) &
		\text{ if } x \in (n,n+\frac{1}{2}).
	\end{cases}
\end{equation}
If $\alpha<0$, we define $g_{\alpha}$ to be the graph obtained by reflecting $g_{-\alpha}$ in
the $x$-axis.
By \cite{lapcmc}[Propositions 5.3 and 5.5], for any $\alpha$ we have 
\begin{equation}
	\label{eqn:graph_height}
	\sgn(\alpha)\l{i_k}=g_\alpha\left(k-\frac{1}{2}\right), k=1, \dots, n, \, 
	\text{ and }\, 
	\sgn(\alpha)l_{\alpha}^{\infty}:=
	\inner{\mu(J)}{\alpha^{\vee}} = g_{\alpha}\left(n+\frac{1}{2}\right).
\end{equation}

\begin{example}
	\label{example:graph}
	We continue Example \ref{example:folded_lambda_chain}. 
	The graphs of $g_{\alpha_2}$ and $g_{\theta}$ are given in Figure \ref{fig:graph}.
\end{example}
	\begin{figure}[h]
        \centering
		\includegraphics[scale=.45]{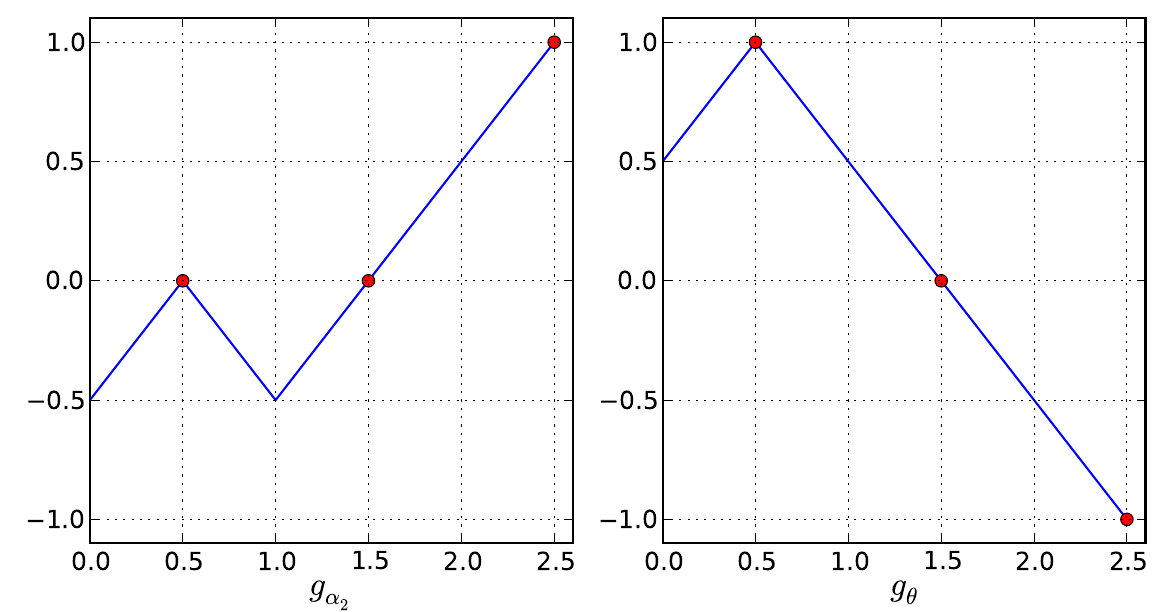}
		\caption{}
		\label{fig:graph}
	\end{figure}

Let $J$ be an admissible subset. 
Let $\delta_{i,j}$ be the Kronecker delta function.
Fix $p$ in $\{0,\ldots,r\}$, so $\alpha_p$ is a simple root if $p>0$, or $\theta$ if $p=0$.
Let $M$ be the maximum of $g_{{\alpha}_p}$. Let $m$ be the minimum index 
$i$  in $\widehat{I}_{{\alpha}_p}(\Delta)$ for which we have $\sgn({\alpha}_p)\l{i}=M$. 
By Proposition \ref{prop:rootF}, if $M\ge\delta_{p,0}$, then we have either $m\in J$ or $m=\infty$; furthermore, if $M>\delta_{p,0}$, then $m$ has a predecessor $k$ in $\widehat{I}_{{\alpha}_p}$, and we have $k\not\in J$. We define
\begin{equation}
	\label{eqn:rootF} 
f_p(J):= 
	\begin{cases}
		(J \backslash \left\{ m \right\}) \cup \{ k \} & \text{ if $M>\delta_{p,0} $ } \\
				\Bzero & \text{ otherwise }.
	\end{cases}
\end{equation}
Now we define $e_p$. Again let $M:= \max g_{{\alpha}_p}$. Assuming that $M>\inner{\mu(J)}{{\alpha}_p^{\vee}}$, let $k$ be the maximum index 
$i$  in $I_{{\alpha}_p}$ for which we have $\sgn({\alpha}_p)\l{i}=M$,
and let $m$ be the successor of $k$ in $\widehat{I}_{{\alpha}_p}$. Assuming also that $M\ge\delta_{p,0}$, by Proposition \ref{prop:rootE} we have $k\in J$, and either $m\not\in J$ or $m=\infty$. 
Define 
\begin{equation}
	\label{eqn:rootE}
e_p(J):= 
	\begin{cases}
		(J \backslash \left\{ k \right\}) \cup \{ m \} & \text{ if }
		M>\inner{\mu(J)}{{\alpha}_p^{\vee}} \text{ and } M \geq \delta_{p,0}  \\
				\Bzero & \text{ otherwise. }
	\end{cases}
\end{equation}
In the above definitions, we use the
convention that $J\backslash \left\{ \infty \right\}= J \cup \left\{ \infty \right\} = J$. 
\begin{example}
	\label{example:root_op}
	We continue Example \ref{example:graph}. 
 We find $f_2(J)$ by noting that $\widehat{I}_{\alpha_2}=\left\{1,4,\infty  \right\}$. 
 From $g_{\alpha_2}$ in Figure \ref{fig:graph} we can see that the heights $\l{i}$ and $l_{\alpha_2}^\infty$ corresponding to these positions are $0,0,1$, 
		 so $k=4,\,m=\infty$, and
		  $f_2(J)=J \cup \{ 4 \} = \left\{ 1,2,3,4,5 \right\}$. 
		  We can see from Figure \ref{fig:graph} that the maximum of 
		  $g_{\theta}=1$, hence $f_0(J)=\Bzero$.
          To compute $e_0(J)$ observe that 
          $\widehat{I}_{\theta}=\left\{3,6\right\}$ with $k = 3$ and
          $m = 6$.
          So $e_0(J) = ( J \backslash \{k\} ) \cup \{ m \} = \left\{ 1,2,5,6 \right\} $.
\end{example}

The following theorem is one of our main results, and will be proved in Section \ref{subsection:propositions}.

\begin{theorem}
    \label{theorem:admissible}
    $\!\!\!\!\!\!\!\!\!\!\!$ \begin{enumerate}
    \item[{\rm (1)}]  If $J$ is an admissible subset and if $f_p(J) \ne \Bzero$, then $f_p(J)$ is also an admissible subset. Similarly for $e_p(J)$. Moreover, $f_p(J)=J'$ if and only if  $e_p(J')=J$.
    \item[{\rm (2)}] We have $ \mu(f_p(J)) = \mu(J) - \alpha_p $. Moreover, if $M\ge\delta_{p,0}$, then
    \[\varphi_p(J)=M-\delta_{p,0}\,,\;\;\;\;\varepsilon_p(J)=M-\langle\mu(J),\alpha_p^\vee\rangle\,,\]
    while otherwise $\varphi_p(J)=\varepsilon_p(J)=\Bzero$. 
    \end{enumerate}
         \end{theorem}

\subsection{Proofs}
\label{subsection:propositions}
In this section we collect necessary results for the proof of Theorem
\ref{theorem:admissible}. The techniques are similar to those in \cite{lapcmc}; we go in detail over the parts of the proofs where there are notable differences, and we refer to the mentioned paper for the remaining parts. 
	\begin{lemma}
	\label{lemma:admissible}
	Let $w\in W$, $\alpha$ a simple root or $\theta$,
	and $\beta$ a positive root. 
	If we have $w \longrightarrow ws_{\beta}$, as well as 
	$w^{-1}(\alpha)>0$ and $s_{\beta}w^{-1}(\alpha)<0$, 
	then $w^{-1}(\alpha)=\beta$.
	\end{lemma}
	\begin{proof}
		If $s_{\alpha}w=ws_{\beta}$, then $w^{-1}(\alpha)=\pm \beta$, 
		and $w^{-1}(\alpha)>0$ implies $w^{-1}(\alpha)=\beta$.
		Suppose by way of contradiction that $s_{\alpha}w\ne ws_{\beta}$.
		First suppose $\alpha$ is a simple root.
		Since $w^{-1}(\alpha)>0 $, then $w \lessdot s_{\alpha}w$. 
		By assumption we have $w \longrightarrow ws_{\beta}$, hence by 
		Lemma \ref{prop:deodhar}
		we have $ws_{\beta} \lessdot s_{\alpha}ws_{\beta}$. 
		But $s_{\beta}w^{-1}(\alpha)<0$ implies 
		$s_{\alpha}ws_{\beta}\lessdot ws_{\beta}$, which is a contradiction. 

		Suppose $\alpha=\theta$. Since $w^{-1}(\theta)>0$, we deduce $w \qstep s_{\theta}w$, by Lemma \ref{lemma:theta}, and then $ws_{\beta} \qstep s_{\theta}ws_{\beta}$, by Lemma  \ref{prop:deodhar0}. 
		On another hand, since $s_{\beta}w^{-1}(\theta)<0$, Lemma \ref{lemma:theta} implies that 
		$s_{\theta}ws_{\beta } \qstep ws_{\beta} $,  which is a
        contradiction.
	\end{proof}
	\begin{lemma}
		\label{lemma:admissible2}
		Let $J = \left\{ j_1 < j_2 < \cdots < j_s \right\} $ 
        be an admissible subset. Assume that $r_{j_a}\dots r_{j_1}(\alpha)>0$ and 
		$r_{j_{b}}\dots r_{j_1}(\alpha)<0$  where
		$\alpha$ is a simple root or $\theta$, and $0\leq a < b$
		(if $a=0$, then the first condition is void). Then there exists $i$ with $a \leq i < b$ such that 
		$\gamma_{j_{i+1}}=\alpha$.
	\end{lemma}
		\begin{proof}
			Find $i$ with $a \leq i < b$ such that $r_{j_i}\dots r_{j_1}(\alpha)>0$ and 
			$r_{j_{i+1}}\dots r_{j_1}(\alpha)<0$. By Lemma \ref{lemma:admissible}, we have 
			$\beta_{j_{i+1}}= r_{j_i}\dots r_{j_1}(\alpha)$. This means that 
			$\gamma_{j_{i+1}}=r_{j_1}\dots r_{j_i}(\beta_{j_{i+1}})=\alpha$.
		\end{proof}

	\begin{proposition}
		\label{prop:A}
		Let $J = \left\{ j_1 < j_2 < \cdots < j_s \right\} $
        be an admissible subset. Assume that $\alpha$ is a simple root or $\theta$, 
		with $I_{\alpha}\ne \emptyset$. 
		Let $m \in I_{\alpha}$ be an element for which its predecessor $k$ 
		(in $I_{\alpha}$) 
		satisfies $(\gamma_k,\varepsilon_k)\in \{(\alpha,1),(-\alpha,-1)\}.$ Then we have $\gamma_m=\alpha$.
	\end{proposition}
		\begin{proof}
			First suppose that $(\gamma_k,\varepsilon_k)=(\alpha,1)$. 
			Assume that $\gamma_m=-\alpha$. 
			Let us define the index $b$ by the condition 
			$j_b < m \leq j_{b+1}$ 
			(possibly $b=s$, in which case the second inequality is dropped).
			We define the index 
			$a$ by the condition $j_{a} < k < j_{a+1}$ 
			(possibly $a=0$, in which case the first inequality is dropped). 
			We clearly have
			$r_{j_1}\dots r_{j_b}(\beta_m)=-\alpha$, 
			which implies $r_{j_b}\dots r_{j_1}(\alpha)<0$. 
			We also have 
			$r_{j_1}\dots r_{j_a}(\beta_k)=\alpha$, so 
			$r_{j_a}\dots r_{j_1}(\alpha)>0$ (hence $a<b$).
			Note that if $\alpha=\theta$, then $a>0$.
			We can now apply Lemma \ref{lemma:admissible2} 
			to conclude that $\gamma_{j_i} =\alpha$ for some $i \in [a+1,b]$.
			Since $k<j_{a+1} \leq j_{b} < m$, we 
			contradicted the assumption that $\gamma_k$ is the predecessor
			of $\gamma_m$ in $I_{\alpha}$.

			Now suppose that $(\gamma_k,\varepsilon_k)=(-\alpha,-1)$. 
			Assume that $\gamma_m = -\alpha$
			and define $b$ as in the previous case. Again we have  $r_{j_b}\dots r_{j_1}(\alpha)<0$.
			Define $a$ by the condition $j_a = k < j_{a+1}$. 
			Hence $r_{j_1}\dots r_{j_{a-1}}(\beta_{j_a})=-\alpha$, so 
			 $r_{j_a}\dots r_{j_1}(\alpha)>0$. This leads to a contradiction by a similar reasoning to the one above.
		\end{proof}

	\begin{proposition}
		\label{prop:A1}
		Let $J$
        be an admissible subset. Assume that $\alpha$ is a simple root for 
		which $I_{\alpha}\ne \emptyset$. Let $m \in I_{\alpha}$ be 
		the minimum of $I_{\alpha}$. Then we have $\gamma_m=\alpha$.
	\end{proposition}
	\begin{proof}
		The proof of Proposition \ref{prop:A} carries through with
		$a=0$.
	\end{proof}
\begin{proposition}
		\label{prop:B}
		Let $J = \left\{ j_1 < j_2 < \cdots < j_s \right\} $
        be an admissible subset. Assume that $\alpha$ is a simple root or $\theta$. Suppose that 
		$I_{\alpha}\ne \emptyset$, and 
		$(\gamma_m,\varepsilon_m) \in\{ (\alpha,1), (-\alpha,-1) \}$ for $m=\max I_{\alpha}.$ Then we have 
		$\langle \gamma_{\infty},\alpha^{\vee} \rangle >0$.
	\end{proposition}
		\begin{proof}
			Assume that the conclusion fails, which means that 
			$r_{j_s}\dots r_{j_1}(\alpha)<0$. 
			First suppose that $(\gamma_m,\varepsilon_m)=(\alpha,1)$. 
			Define the index $a$ 
			by the condition $j_a < m < j_{a+1}$. (If $a=0$ or $a=s$ 
			one of the two inequalities is dropped). We have 
			$r_{j_1}\dots r_{j_{a}}(\beta_m)=\alpha$, so 
			$r_{j_a}\dots r_{j_1}(\alpha)>0$ (hence $a \ne s$). 
			Note that if $\alpha=\theta$, then $a>0$.
			We now apply Lemma \ref{lemma:admissible2} 
			to conclude that $\gamma_{j_i}=\alpha$ for $i \in [a+1,s]$.
			Since $m<j_{a+1}\leq j_s$, this contradicts that $m=\max I_{\alpha}$. 

			Now suppose that $(\gamma_m,\varepsilon_m)=(-\alpha,-1)$.
			In this case we define the index $a$ by $j_a = m <j_{a+1}$. We have  
			$r_{j_1}\dots r_{j_{a-1}}(\beta_{j_a})=-\alpha$, so 
			$r_{j_a}\dots r_{j_1}(\alpha)>0$. This leads to a contradiction by a similar reasoning to the one above.
		\end{proof}

\begin{proposition}
		\label{prop:B1}
	Let $J$
    be an admissible subset. Assume that, for some simple root $\alpha$, we have $I_{\alpha}= \emptyset$.
		Then  
		$\langle \gamma_{\infty},\alpha^{\vee} \rangle >0$.
	\end{proposition}
	\begin{proof}
		The proof of Proposition \ref{prop:B} carries through with $a=0$.	
	\end{proof}


	Let us now fix a simple root $\alpha$.
	We will rephrase some of the above results in a simple way in terms of $g_{\alpha}$, and we will deduce some consequences.
	Assume that $I_{\alpha}=\left\{ i_1 < i_2 < \dots < i_n \right\}$, so that $g_{\alpha}$ is defined 
	on $[0,n+\frac{1}{2}]$, and let $M$ be the maximum of $g_{\alpha}$. 
	Note first that the function $g_{\alpha}$ is determined by the sequence $(\sigma_1, \dots, \sigma_{n+1})$, where 
	$\sigma_j = (\sigma_{j,1},\sigma_{j,2}):= (\sgn(\gamma_{i_j}), \varepsilon_{i_j}\sgn (\gamma_{i_j}))$ for 
	$1\leq j\leq n$,  and $\sigma_{n+1}= \sigma_{n+1,1}:=\sgn (\langle \gamma_{\infty}, \alpha^{\vee} \rangle)$. From 
	Propositions \ref{prop:A}, \ref{prop:A1}, \ref{prop:B} and \ref{prop:B1} we have the following restrictions.
	\begin{enumerate}[(C1)]
	\item   $\sigma_{1,1}=1$.  
	\item	$\sigma_{j,2}=1 \Rightarrow \sigma_{j+1,1}=1$.
	\end{enumerate}
\begin{proposition}
			\label{prop:main1}
	  If $g_{\alpha}(x)=M$, then $M \in \Z_{\geq 0}$, $x=m+\frac{1}{2}$ for $0 \leq m \leq n$, 
	  and $\sigma_{m+1} \in \left\{ (1,-1),1 \right\}$.
\end{proposition}
		\begin{proof}
			By (C1), we have $M\geq 0$, therefore
			$g_{\alpha}(0) = - \frac{1}{2} \ne M$. 
			For $m \in \left\{ 1, \dots, n \right\}$,  
			if $g_{\alpha}(m)=M$ then $\sigma_{m,2}=1$, and (C2) leads to a contradiction.
			The last statement is obvious.
		\end{proof}
		We use Proposition \ref{prop:main1} implicitly in the proof of 
		Proposition \ref{prop:main2} and Proposition \ref{prop:main3}.
	\begin{proposition}
		\label{prop:main2}
		Assume that $M>0$, and let $m$ be such that $m+ \frac{1}{2} = \min g_{\alpha}^{-1}(M)$. We have $m>0$,
		$\sigma_m = (1,1)$, and $g_{\alpha}(m-\frac{1}{2})= M-1$. Moreover, we have 
		$g_{\alpha}(x) \leq M-1$ for $0 \leq x \leq m- \frac{1}{2}$. 
	\end{proposition}
		\begin{proof}
			By (C1) we have $g_{\alpha}(\frac{1}{2})=0$, so $m>0$. 
			If 
			$\sigma_m \in \{(-1,-1), (1,-1)\}$, then we have 
			$g_{\alpha}(m-\frac{1}{2})=M$, 
			which contradicts the definition of $m$. 
			If $\sigma_m=(-1,1)$, then $g_{\alpha}(m-1)=M-\frac{1}{2}>-\frac{1}{2}$. By (C1) we have $m \geq 2 $, and by (C2) 
			we have $\sigma_{m-1,2}=-1$. This implies that $g_{\alpha}(m-\frac{3}{2})=M$, contradicting the definition of $m$. Hence $\sigma_m=(1,1)$.

			Suppose by way of contradiction that the last statement in the corollary fails. Then there exists a $k$ with
			$1 \leq k \leq m-1$ such that $g_{\alpha}(k-1)=M-\frac{1}{2}>-\frac{1}{2}$ and 
			$\sigma_{k,1}=-1$. 
			Condition (C1) implies that $k \geq 2$ and Condition (C2) implies
			$\sigma_{k-1,2}=-1$. 
			This implies $g_{\alpha}(k-\frac{3}{2})=M$, contradicting the definition of $m$.
		\end{proof}

	\begin{proposition}
		\label{prop:main3}
		Assume that $M> g_{\alpha}(n+\frac{1}{2})$, and let $k$ be such that $k-\frac{1}{2}= \max 
		g_{\alpha}^{-1}(M)$.  We have $k \leq n, \sigma_{k+1} \in \{ (-1, -1), -1\}$, and  $g_{\alpha}(k+\frac{1}{2})
		=M-1$. Moreover, we have $g_{\alpha}(x) \leq M-1$  for 
		$k + \frac{1}{2} \leq x \leq n+ \frac{1}{2}$.
	\end{proposition}
	\begin{proof}
	Since $	M> g_{\alpha}(n+\frac{1}{2})$, it follows that $k\leq n$.
	If $\sigma_{k+1} \in \{ (1,1),(1,-1),1 \}$ then $g_{\alpha}(k+\frac{1}{2})=M$,
	contradicting the choice of $k$. If $\sigma_{k+1}=(-1,1)$, then by (C2) we have 
	$\sigma_{k+2,1}=1$, and $g_{\alpha}(k+\frac{3}{2})=M$, contradicting the choice of 
	$k$. Hence $\sigma_{k+1} \in \{(-1,-1),-1 \}$.

	Suppose by way of contradiction the last statement in the corollary fails. Then there
	exists an $m$ with $k+2 \leq m \leq n$ such that $g_{\alpha}(m)=M-\frac{1}{2}$
	and $\sigma_{m,2}=1$. Condition (C2) implies that $\sigma_{m+1,1}=1$, so 
	$g_{\alpha}(m+\frac{1}{2})=M$, contradicting the choice of $k$.
	\end{proof}

We now consider $g_{\theta}$. Since $\theta<0$, the definition of the piecewise linear function $g_\theta$ requires us to define its linear steps by 
$\sigma_j = (\sigma_{j,1},\sigma_{j,2}):= (-\sgn(\gamma_{i_j}),-\varepsilon_{i_j}\sgn (\gamma_{i_j}))$ for $1 \leq j \leq n$, and 
$\sigma_{n+1}=\sigma_{n+1,1}:=\sgn (\langle \gamma_{\infty}, \theta^{\vee} \rangle)$. 
From Propositions \ref{prop:A} and \ref{prop:B} we conclude that condition (C2) holds for 
$g_{\theta}$. We can replace condition (C1) by restricting to admissible subsets
$J$ where $M$ is large enough, as we will now explain.
In the proof of Proposition \ref{prop:main1}, condition (C1) is needed to conclude that 
$g_{\alpha}(0) \ne M$. It is possible that $g_{\theta}(0)=M$, but if we restrict to
$g_{\theta}$ where $M\geq 1$ we can conclude that $g_{\theta}(0)=\frac{1}{2} \ne M$, and the
rest of the proof follows through.
In the proof of Proposition \ref{prop:main2}, condition (C1) first implies $m\ge 1$; we 
can conclude the same thing if we assume $M\ge 2$, since $g_\theta(\frac{1}{2})\le 1$. Then we
need to derive $m\ge 2$ from $g_\theta(m-1)=M-\frac{1}{2}$; again, if $M\ge 2$, then $M-\frac{1}{2}>g_\theta(0)=\frac{1}{2}$, so $m-1>0$. 
Note that Proposition \ref{prop:main3} depends on Proposition \ref{prop:main1} so we need to assume $M \geq 1$ here too. We have therefore proved the following propositions.

\begin{proposition}
	\label{prop:main1_theta}
	  Suppose $M\geq 1$. If $g_{\theta}(x)=M$, then $M \in \Z_{\geq 1}$, $x=m+\frac{1}{2}$ for $0 \leq m \leq n$,  
	  and $\sigma_{m+1} \in \left\{ (1,-1),1 \right\}$.
\end{proposition}
	\begin{proposition}
		\label{prop:main2_theta}
		Assume that $M\geq2$, and let $m$ be such that 
		$m+ \frac{1}{2} = \min g_{\theta}^{-1}(M)$. We have $m>0$,
		$\sigma_m = (1,1)$, and $g_{\theta}(m-\frac{1}{2})= M-1$. Moreover, we have 
		$g_{\theta}(x) \leq M-1$ for $0 \leq x \leq m- \frac{1}{2}$. 
\end{proposition}

	\begin{proposition}
		\label{prop:main3_theta}
		Assume $M\geq 1$, and also that $M> g_{\theta}(n+\frac{1}{2})$. Let $k$ be such that $k-\frac{1}{2}= \max 
		g_{\theta}^{-1}(M)$.  We have $k \leq n, \sigma_{k+1} \in \{ (-1, -1), -1\}$, and  $g_{\theta}(k+\frac{1}{2})
		=M-1$. Moreover, we have $g_{\theta}(x) \leq M-1$  for 
		$k + \frac{1}{2} \leq x \leq n+ \frac{1}{2}$.
	\end{proposition}

Recall from from Section \ref{subsection:crystalop} the definitions of the finite sequences
$I_{\alpha}(\Delta)$ and $\widehat{I}_{\alpha}(\Delta)$, where $\alpha$ is a root, 
of $g_{\alpha}$, as well as the related notation.

Fix $p$, so $\alpha_p$ is a simple root if $p>0$, or $\theta$ if $p=0$.
Let $M$ be the maximum of $g_{\alpha_p}$, and suppose that $M \geq \delta_{p,0}$.
Note this is always true for $p \ne 0$ by Proposition \ref{prop:main1}.
Let $m$ be the minimum index 
$i$  in $\widehat{I}_{\alpha_p}(\Delta)$ for which we have $\sgn(\alpha_p)\l{i}=M$. 
The following proposition is an immediate consequence of 
Propositions \ref{prop:main1}, \ref{prop:main2}, \ref{prop:main1_theta}, \ref{prop:main2_theta}.

\begin{proposition} Given the above setup, the following hold.
	\label{prop:rootF}
    \begin{enumerate}
        \item[{\rm (1)}] If $m \ne \infty$, then $\gamma_m=\alpha_p$ and $m \in
            J$. \label{rootFa}

        \item[{\rm (2)}] If $M>\delta_{p,0}$, then 
            $m$ has a predecessor $k$ in $\widehat{I}_{\alpha_p}(\Delta)$ such that 
            \[
                \gamma_k=\alpha_p,\, k \not \in J, \,  \mbox{and }\, \sgn(\alpha_p)\l{k} = M-1.
            \]\label{rootFb}
    \end{enumerate}
\end{proposition}

Now assume that $M>\inner{\mu(J)}{{\alpha}_p^{\vee}}$. Let $k$ be the maximum index 
$i$  in $I_{\alpha_p}(\Delta)$ for which we have $\sgn(\alpha_p)\l{i}=M$, 
and let $m$ be the successor of $k$ in $\widehat{I}_{\alpha_p}(\Delta)$. 
The following analogue of Proposition \ref{prop:rootF} is proved in a
similar way, based on Propositions \ref{prop:main1}, \ref{prop:main3}, \ref{prop:main1_theta}, \ref{prop:main3_theta}.

\begin{proposition} Given the above setup, and assuming also that $M\ge\delta_{p,0}$, the following hold.
	\label{prop:rootE}
    \begin{enumerate}
        \item[{\rm (1)}] We have $\gamma_k=\alpha_p$ and $k \in J$.
        \item[{\rm (2)}] If $m\ne \infty$, then 
            \[\gamma_m=-\alpha_p,\,  m \not \in J,\, \mbox{ and } \sgn(\alpha_p)\l{m} = M-1.\]
    \end{enumerate}
\end{proposition}

\begin{proof}[Proof of Theorem {\rm \ref{theorem:admissible}}]
			Suppose $p \ne 0$. We consider $f_p$ first. 
			The cases corresponding to $m\ne \infty$ and $m=\infty$ can be proved in similar ways, 
			so we only consider the first case. Let $J=\left\{ j_1 < j_2< \ldots < j_s \right\}$, and let
			$w_i := r_{j_1}r_{j_2} \dots r_{j_i}$.
			Based on Proposition \ref{prop:rootF}, let $a<b$ be such that 
			\[ j_a < k < j_{a+1} < \dots < j_b = m < j_{b+1} \;; \]
			if $a=0$ or $b+1>s$, then the corresponding indices $j_a$, respectively $j_{b+1}$, are missing.
			To show that $(J \backslash \left\{ m \right\}) \cup \left\{ k \right\}$ 
			is an admissible subset,
			it is enough to prove that we have the path in the quantum Bruhat graph 
			\begin{equation}
				w_a \longrightarrow w_ar_{k} 
				\longrightarrow w_ar_k r_{j_{a+1}}
				\longrightarrow
				\dots
				\longrightarrow w_ar_k r_{j_{a+1}}\dots r_{j_{b-1}}
				=
			w_b\,.
			\label{newadmissible}
			\end{equation} 
			By our choice of $k$, we have  
			\begin{align}
		w_a(\beta_k)=\alpha_p \;\iff\; w_a^{-1}(\alpha_p)=\beta_k >0 
	\;\iff\; w_a \lessdot s_pw_a = w_ar_k\,.
				\label{admissiblebasecase}
			\end{align}
			So we can rewrite (\ref{newadmissible}) as
			\begin{equation}
				w_a \longrightarrow s_pw_a
				\longrightarrow s_pw_{a+1}
				\longrightarrow
				\dots
				\longrightarrow s_pw_{b-1}
				=
			w_b\,.
			\label{newadmissibleb}
			\end{equation} 

			We will now prove that (\ref{newadmissibleb}) is a path in the quantum Bruhat graph.
		Observe that, for $a<i\le b$, we have
			\begin{align*}
				s_p w_{i-1}=w_i=w_{i-1}r_{j_i} \;\iff\; w_{i-1}(\beta_{j_i})=\pm \alpha_p
				\;\iff\; j_i \in I_{\alpha}\,.
			\end{align*}
			Our choice of $k$ and $b$ implies that we have 
			\begin{equation}
				s_p w_{i-1} \ne w_i \; \text{ for } a<i<b\,,\;\;\text{ and } s_p w_{b-1}=w_b\,.
				\label{equation:tempinductionstep}
			\end{equation}
			Since $J$ is admissible, we have 
			\begin{equation}
				w_{i-1} \longrightarrow  w_i\,. 
				\label{eqn:diamond_lower}
			\end{equation}
			Starting from (\ref{admissiblebasecase}), and then using \eqref{equation:tempinductionstep}-\eqref{eqn:diamond_lower}, we can apply Lemma \ref{prop:deodhar} repeatedly to conclude that
			\begin{equation}
				\label{eqn:diamond_upper}
				s_pw_{i-1} \longrightarrow s_pw_{i} \;\text{ and }\; w_{i} \lessdot s_pw_{i}\,,\;\;
				\text{ for } a<i<b\,.
			\end{equation}

	The proof for $e_p(J)$ is similar. The main difference is that we need the ``if'' part of Lemma \ref{prop:deodhar}, whereas above we used the ``only if'' part. 

		The above proof follows through for $p=0$, based on 
		Lemma \ref{lemma:theta}, which is used to derive the analogue of \eqref{admissiblebasecase}, and Lemma \ref{prop:deodhar0}, which replaces Lemma \ref{prop:deodhar}.
		
		We can prove that $f_p(J)=J'$ if and only if  $e_p(J')=J$ based on \cite{lapcmc}[Proposition 7.4 (1)]; this still holds in the above setup (for any $p$, including $p=0$), based on Propositions \ref{prop:main1}$-$\ref{prop:main3_theta}. The same result can be invoked to derive the formulas for $\varphi_p(J)$ and $\varepsilon_p(J)$. 
		
		In order to show that $f_p$ changes weights by $-\alpha_p$, the proof of \citep{lapcmc}[Proposition 7.1 (3)] can be applied
in our context. In essence, we note that $\mu(f_p(J))$ is $-\widehat{ t}_k  \widehat{ t}_m (-\mu)$ if $ m \ne \infty$,
and $-\widehat{ t}_k (-\mu) $ otherwise, where $ \widehat{ t }_j:=s_{|\gamma_j|, -\l{j}}=s_{\gamma_j,-{\rm sgn}(\gamma_j)l_j^J}$. By Proposition \ref{prop:rootF} and \eqref{eqn:graph_height}, we have
$\widehat{ t }_k=s_{\alpha_p,-M}$ and $\widehat{ t }_m=s_{\alpha_p,-(M-1)}$. The rest of the calculation is the same as in the proof mentioned above.
		\end{proof}

\subsection{Main application}\label{mainappl}

We summarize the main results in \cite{unialcmod2}, cf. also \cite{unialcmod, lnseda}. The setup is that of untwisted affine root systems. 

\begin{theorem}{\rm \cite{unialcmod2}}
    \label{mainconj} Consider a composition ${\mathbf{p}}=(p_1,\ldots,p_k)$ and the corresponding crystal $B:=\bigotimes_{i=1}^{k} B^{p_i,1}$. Let $\lambda:=\omega_{p_1}+\ldots+\omega_{p_k}$, and let $\Gamma_{\rm lex}$ be a corresponding lex $\lambda$-chain (see above). 
    
        {\rm (1)}  The (combinatorial) crystal $\A(\Gamma_{\rm lex})$ is isomorphic to the subgraph of $B$ consisting of the dual Demazure arrows, via a specific bijection. 

        {\rm (2)}  If the vertex $b$ of $B$ corresponds to $J$ under the isomorphism in part {\rm (1)},  then 
         the energy is given by $D_B(b)-C=-\height(J)$, where $C$ is a global constant.
  \end{theorem}

\begin{remarks}\label{ndd} (1) The entire crystal $B$ is realized in \cite{unialcmod2} in terms of the so-called {\em quantum LS path model}. If we identify the two, the bijection in Theorem \ref{mainconj} (1) is the ``forgetful map'' from the quantum alcove model to the quantum LS path model, so it is a very natural map. Therefore, we think of the former model as a mirror image of the latter, via this bijection. However, if we use this identification to construct the non-dual Demazure arrows in the quantum alcove model, we quickly realize that, in general, the constructions  are considerably more involved than \eqref{eqn:rootF}-\eqref{eqn:rootE}, cf. Remark \ref{nddA} (1) and Example \ref{nddAex}. 

(2) Although the quantum alcove model so far misses the non-dual Demazure arrows, it has the advantage of being a discrete model. Therefore, combinatorial methods are applicable, for instance in proving the independence of the model from the choice of an initial alcove path (or $\lambda$-chain of roots), see below, including the application in Remark \ref{perfectcase} (2). This should be compared with the continuous arguments used for the similar purpose in the Littelmann path model \cite{litpro}. 

(3) Theorem \ref{mainconj}, combined with the Ram-Yip formula for Macdonald polynomials \cite{raycfm}, implies that the graded character of a tensor product of column shape KR modules (the grading being by the energy function) concides with the corresponding Macdonald polynomial specialized at $t=0$ \cite{unialcmod2}. 
\end{remarks}

Based on Theorem \ref{mainconj} (1), as well as on the realization of the same subgraph of $B$ in types $A$ and $C$ in terms of a different $\lambda$-chain (see Theorems \ref{theorem:crystal_isomorphism} and \ref{theorem:crystal_isomorphismTypeC}), we make the following conjecture.

\begin{conjecture}\label{mainmainconj}
 Theorem {\rm \ref{mainconj}} holds for any choice of a $\lambda$-chain (instead of a lex $\lambda$-chain). 
\end{conjecture}

We plan to prove this conjecture in \cite{lalurc} by using Theorem \ref{mainconj} as the starting point. Then, given two $\lambda$-chains $\Gamma$ and $\Gamma'$, we would construct a bijection between $\A(\Gamma)$ and $\A(\Gamma')$ preserving the dual Demazure arrows and the height statistic; this would mean that the quantum alcove model does not depend on the choice of a $\lambda$-chain. This construction will be based on generalizing to the quantum alcove model the so-called Yang-Baxter moves in \cite{lenccg}. As a result, we would obtain a collection of a priori different bijections between $B$ and $\A(\Gamma)$.

\begin{remarks}\label{perfectcase} (1) We believe that the bijections mentioned above would be identical. In fact, this would clearly be the case if all the tensor factors of $B$ are perfect crystals, see Section \ref{subsection:KR-crystals}. Indeed, then the subgraph of $B$ consisting of the dual Demazure arrows is connected, so there is no more than one isomorphism between it and $\A(\Gamma)$.

(2) In the case when all the tensor factors of $B$ are perfect crystals, a corollary of the work in \cite{lalurc} would be the following application of the quantum alcove model, cf. Remark \ref{perfectcase} (1). By making specific choices for the $\lambda$-chains $\Gamma$ and $\Gamma'$, the bijection between $\A(\Gamma)$ and $\A(\Gamma')$ mentioned above would give a uniform realization of the combinatorial $R$-matrix (i.e., the unique affine crystal isomorphism commuting factors in a tensor product of KR crystals). In fact, we believe that this statement would hold in full generality, rather than just the perfect case.
\end{remarks}

\section{The quantum alcove model in types \texorpdfstring{$A$ and $C$}{A and C}}

In this section we specialize the quantum alcove model to types $A$ and $C$, and prove that the bijections
constructed in \citep{Lenart}, from the objects of the specialized quantum alcove
model to the tensor products of the corresponding KN columns (see Section \ref{subsection:KR-crystals}), are affine crystal isomorphisms.

\label{modelac}
\subsection{Type \texorpdfstring{$A$}{A}}
\label{subsection:TypeA}
We start with the basic facts about the root system of type $A_{n-1}$. 
We can identify the space $\hh_\R^*$ with the quotient $V:=\R^n/\R(1,\ldots,1)$,
where $\R(1,\ldots,1)$ denotes the subspace in $\R^n$ spanned 
by the vector $(1,\ldots,1)$.  
Let $\varepsilon_1,\ldots,\varepsilon_n\in V$ 
be the images of the coordinate vectors in $\R^n$.
The root system is 
$\Phi=\{\alpha_{ij}:=\varepsilon_i-\varepsilon_j \::\: i\ne j,\ 1\leq i,j\leq n\}$.
The simple roots are $\alpha_i=\alpha_{i,i+1}$, 
for $i=1,\ldots,n-1$. The highest root $\widetilde{\alpha}=\alpha_{1n}$. We let
$\alpha_0=\theta=\alpha_{n1}$.
The weight lattice is $\Lambda=\Z^n/\Z(1,\ldots,1)$. The fundamental weights are $\omega_i = \varepsilon_1+\ldots +\varepsilon_i$, 
for $i=1,\ldots,n-1$. 
A dominant weight $\lambda=\lambda_1\varepsilon_1+\ldots+\lambda_{n-1}\varepsilon_{n-1}$ is identified with the partition $(\lambda_{1}\geq \lambda_{2}\geq \ldots \geq \lambda_{n-1}\geq\lambda_n=0)$ having at most $n-1$ parts. Note that $\rho=(n-1,n-2,\ldots,0)$.
Considering the Young diagram of the dominant weight $\lambda$ as a concatenation of columns, 
whose heights are $\lambda_1',\lambda_2',\ldots$, 
corresponds to expressing $\lambda$ as $\omega_{\lambda_1'}+\omega_{\lambda_2'}+\ldots$ 
(as usual, $\lambda'$ is the conjugate partition to $\lambda$). 

The Weyl group $W$ is the symmetric group $S_n$, which acts on $V$ by permuting the coordinate vectors $\varepsilon_1,\ldots,\varepsilon_n$. Permutations $w\in S_n$ are written in one-line notation $w=w(1)\ldots w(n)$. For simplicity, we use the same notation $(i,j)$ with $1\le i<j\le n$ for the root $\alpha_{ij}$ and the reflection $s_{\alpha_{ij}}$, which is the transposition $t_{ij}$ of $i$ and $j$. We recall a criterion for the edges of the type $A$ quantum Bruhat graph. We need the circular order $\prec_i$ on $[n]$ starting at $i$, namely $i\prec_i i+1\prec_i\ldots \prec_i n\prec_i 1\prec_i\ldots\prec_i i-1$. It is convenient to think of this order in terms of the numbers $1,\ldots,n$ arranged on a circle clockwise. We make the convention that, whenever we write $a\prec b\prec c\prec\ldots$, we refer to the circular order $\prec=\prec_a$.
	  
	\begin{proposition}{\rm \cite{Lenart}}
		\label{prop:quantum_bruhat_order_type_A}
	For $1\leq i<j\leq n$, we have an edge $w \stackrel{(i,j)}{\longrightarrow} w(i,j)$ if and only if 
	there is no $k$ such that $i<k<j$ and $w(i) \prec w(k) \prec w(j)$.
	\end{proposition}
	\begin{example}\label{qbgex}
	The quantum Bruhat graph of type $A_2$, i.e., on the symmetric group $S_3$, is indicated in Figure \ref{fig:qbg}.
	\begin{figure}[h]
        \centering
		\includegraphics[scale=0.5]{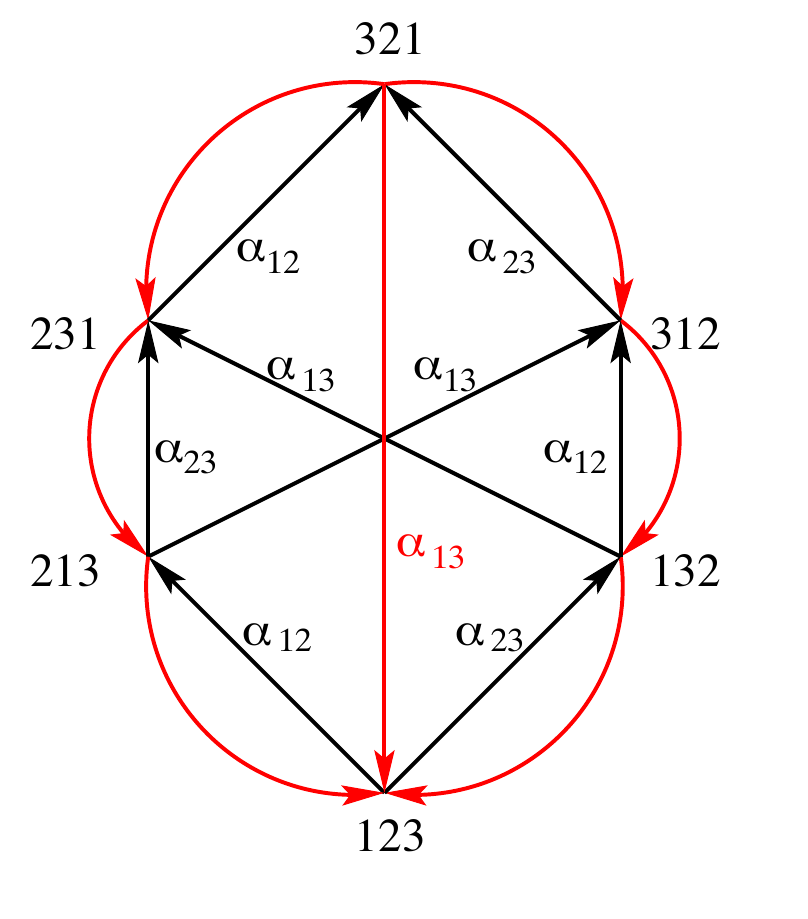}
		\caption{}
		\label{fig:qbg}
	\end{figure}
		\end{example}

We now consider the specialization of the quantum alcove model to type $A$. 
For any $k=1, \ldots , n-1$, we have the following $\omega_k$-chain, from $A_{\circ}$ to 
$A_{-\omega_k}$, denoted by 
	$\Gamma(k)$ \cite{lapawg}:
	\begin{equation}
		\begin{matrix*}[l]
			( (k,k+1),&    (k,k+2)&,     \ldots,&  (k,n),    \\
			\phantom{(} (k-1,k+1),&  (k-1,k+2)&,     \ldots,& (k-1,n),  \\
			\phantom{(,}\vdots & \phantom{,}\vdots & & \phantom{,}\vdots \\
			\phantom{(}(1,k+1),&   (1, k+2)&,      \ldots,&  (1,n)) \,.
		\end{matrix*}
		\label{eqn:lambdachainA}
	\end{equation}
	\begin{example}
		\label{example:broken_column_type_A}
		For $n=4, \Gamma(2)$ can be visualized as obtained from the following broken column, by pairing row numbers in the top and bottom parts in the prescribed order.
		\[
		\tableau{ 1 \\ 2 \\ \\ 3\\  4 }\quad, \quad   
		\Gamma(2)=\{ (2,3),(2,4), (1,3),(1,4) \}\,.
		\]
	Note that the top part of the above broken column corresponds to $\omega_2$.
	\end{example}
	Fix a dominant weight/partition $\lambda$ for the remainder of this section. We construct a $\lambda$-chain $\Gamma=\left( \beta_1, \beta_2, \dots , \beta_m \right)$ as the concatenation 
	$\Gamma:= \Gamma^{1}\dots\Gamma^{\lambda_1}$, where $\Gamma^{j}=\Gamma(\lambda'_j)$.
	Let  $J=\left\{ j_1 < \dots < j_s \right\}$ be a set of folding positions in 
	 $\Gamma$, not necessarily admissible, and let $T$ be the corresponding list of roots of $\Gamma$. 
	 The factorization of $\Gamma$ induces a 
	 factorization of $T$ as 
	 $T=T^{1}T^2 \dots T^{\lambda_1}$, and of $\Delta=\Gamma(J)$ as $\Delta=\Delta^1 \dots \Delta^{\lambda_1}$. 
	 Recalling that the roots in $\Delta$ were denoted $\gamma_k$, we use the notation $\gamma_k\in\Delta^q$ to indicate that the $k$th root in $\Delta$ falls in the segment $\Delta^q$ (rather than the fact that $\Delta^q$ contains a root equal to $\gamma_k$). 
	 We denote by $T^{1} \dots T^{j}$ the permutation obtained by composing the transpositions in $T^{1}, \dots , T^{j}$ left to right. 
For $w\in W$, written $w=w_1w_2\dots w_n$, let $w[i,j]=w_i\dots w_j$. 
	  
	  We now recall from \cite{Lenart} the construction of the correspondence between the type $A$ quantum alcove model and model based on diagram fillings. 
	  
	  \begin{definition}
		  \label{definition:fill}
	  Let $\pi_{j}=\pi_{j}(T):=T^1 \dots T^j$. We define the \emph{filling map},
	  which associates with each $J\subseteq[m]$ a filling of the Young diagram $\lambda$, by 
	  \begin{equation}
		  \label{eqn:filling_map}
		  \fill(J)=\fill(T):=C_{1}\dots C_{\lambda_1}\,,
		  \;\mbox{ where } C_{i}:=\pi_i[1,\lambda'_i].
	  \end{equation}
	    We define the \emph{sorted filling map} $\sfill(J)$  by sorting ascendingly
		  the columns of $\fill(J)$.
	  \end{definition}
	  
	  \begin{example}
		  \label{example:filling_map}
		  Let $n=3$ and $\lambda=(4,3,0)$, which is identified with $4\varepsilon_1 + 3\varepsilon_2 = 3\omega_2 +\omega_1$, and corresponds to the Young diagram $\, \tableau{ { } & { } & { } & { } \\ { } & { } & {}  }$. We have
		  \[\Gamma =\Gamma^1 \Gamma^2 \Gamma^3\Gamma^4 = \Gamma(2)\Gamma(2)\Gamma(2)\Gamma(1) = 
			  \{ \underline{(2,3)},\underline{(1,3)} \,|\, \underline{(2,3)}, (1,3) \,|\, \underline{(2,3)},(1,3)\,|\,\underline{(1,2)},(1,3) \}, \] where we underlined the roots 
			  in positions  
		  $J=\{1,2,3,5,7 \}$. 
		  Then 
		  \[T=
			  \{(2,3),(1,3)\,|\,(2,3)\,|\,(2,3)\,|\,(1,2) \},\,
			  \mbox{ and } 
			  \]
		  \begin{equation}
			  \label{eqn:delta_factorization}
			  \Gamma(J)=\Delta=\Delta^1\Delta^2\Delta^3\Delta^4=
			  \{\underline{(2,3)}, \underline{(1,2)} \, | \, 
			  \underline{(3,1)}, (2,3) \, | \, \underline{(1,3)}, (2,1)\,|\,
			  \underline{(2,3)},(3,1) \}, 
		  \end{equation}
		  where we again underlined the folding positions, and indicated the factorizations of $T$ and $\Delta$ by bars.
		  It is easy to check that $J$ is admissible; indeed, the sequence of permutations \eqref{eqn:admissible} corresponding to $J$ (written as broken columns) is a path in the quantum Bruaht graph, cf. Example \ref{qbgex}: 
		            \begin{equation}
				  \tableau{1 \\ \mathbf{2} \\ \\ \mathbf{3}}
				  \lessdot
				  \tableau{\mathbf{1} \\ 3 \\ \\ \mathbf{2}}
				  \lessdot
				  \tableau{2 \\ 3 \\ \\ 1}
				  \,|\,
				  \tableau{2 \\ \mathbf{3} \\ \\ \mathbf{1}}
				  \qstep
				  \tableau{2 \\ 1 \\ \\ 3}
				  \,|\,
				  \tableau{2 \\ \mathbf{1} \\ \\ \mathbf{3}}
				  \lessdot
				  \tableau{2 \\ 3 \\ \\ 1}
				  \,|\,
				  \tableau{\mathbf{2} \\ \\ \mathbf{3} \\ 1}
				  \lessdot
				  \tableau{3 \\ \\ 2 \\ 1}
				  \,|.
				  \label{eqn:admissible_chain}
			  \end{equation}
			  By considering the top part of the last column in each segment 
			  and by concatenating these columns left to right, we 
			   obtain $\fill(J)$, i.e.,
			 $ \fill(J) = \tableau{2 & 2 &2 & 3 \\ 3 & 1& 3 }\,$.
	  \end{example}

\begin{theorem}{\rm \cite{Lenart}[Theorem 4.1]}
	\label{theorem:bijection_type_A}	
	The map $\sfill$ is a bijection between $\A(\Gamma)$ and $\B$, see {\rm \eqref{btensl}}.
\end{theorem}

We now state the main result of this section.

\begin{theorem}
	\label{theorem:crystal_isomorphism}
	The map $\sfill$ is an affine crystal isomorphism between $\A(\Gamma)$ and the subgraph of $\B$ consisting of the dual Demazure arrows. 
	In other words, given $\sfill(J)=b$, there is a dual Demazure arrow $b\rightarrow f_p(b)$ if and only if $f_p(J)\ne \Bzero$, and we have  $f_i(b)=\sfill(f_i(J))$.
	%
\end{theorem}

\begin{remarks}\label{nddA} (1) The affine crystal isomorphism in Theorem \ref{theorem:crystal_isomorphism} is unique, cf. Remark \ref{perfectcase} (1).  Therefore, this isomorphism gives the unique way to realize the non-dual Demazure arrows in  $\A(\Gamma)$, cf. Remark \ref{ndd} (1) and Example \ref{nddAex} below. 

(2) In \cite{Lenart} it was shown that the map $\sfill$ translates the height statistic to the charge statistic, which is known to express the energy function in the model based on diagram fillings, cf. Theorem \ref{theorem:charge}. This should be compared with Theorem \ref{mainconj} (2), where the constant $C$ is $0$ in this case.
\end{remarks}

\begin{example}\label{nddAex}
In type $A_2$, consider $\lambda=(3,2,0)$, the $\lambda$-chain in Example
	\ref{example:lambda_chain}, and  the admissible subset $J=\{1,2,3,5\}$, cf. Examples \ref{example:folded_lambda_chain}, \ref{example:graph}, and \ref{example:root_op}.
	We have 
	$b=\sfill(J)= \tableau{ 2 & 1 & 1 \\ 3 & 2}$, and 
	$f_0(b)=\sfill(\emptyset)= \tableau{1 & 1 & 1 \\ 2 & 2}$, cf. Example \ref{example:root0}.
	However, $b \to f_0(b)$ is not a dual Demazure arrow, and indeed $f_0(J)=\Bzero$, cf. Example \ref{example:root_op}. 
	In order to realize this arrow in the quantum alcove model, we would have to define $f_0(J)=\emptyset$. 
	This shows that, in general, the changes in an admissible subset corresponding to non-dual Demazure arrows are hard to control.
	Nevertheless, such arrows are sometimes still realized by our construction \eqref{eqn:rootF}, assuming that we drop the corresponding condition $M>1$. For an example, still in type $A_2$, consider $\lambda=(3,0,0)$, the $\lambda$-chain 
	$(\alpha_{12},\alpha_{13},\alpha_{12},\alpha_{13},\alpha_{12},\alpha_{13})$, and $J=\{3,4\}$. We have 
	$b=\sfill(J)= \tableau{1 & 3 & 3}$, and 
	$b\rightarrow f_0(b)= \tableau{1 & 3 & 1}$ is not a dual Demazure arrow. Now note that the corresponding arrow in the quantum alcove model is $J\mapsto J\cup\{5\}$, which is given by the mentioned relaxed version of \eqref{eqn:rootF}.
\end{example}

The main idea of the proof of Theorem \ref{theorem:crystal_isomorphism} is the following. 
The signature of a filling, used to define the crystal operator $f_p$, can 
be interpreted as a graph similar to the graph of $g_{\alpha_p}$, which is used to define the 
crystal operator on the corresponding admissible subsequence. The link between the 
two graphs is given by Lemma \ref{lemma:height_counting} below, 
called the height counting lemma, which we now explain. 

Let $N_c(\sigma)$ denote the number of entries $c$ in a filling $\sigma$. Let  
		$\ct(\sigma)=(N_1(\sigma), \dots, N_n(\sigma))$ be the content of $\sigma$, which is identified with a type $A_{n-1}$ weight.
Let $\sigma[q]$ be the filling consisting of the columns $1,2, \dots , q$ of $\sigma$.
Given a $\lambda$-chain and a corresponding sequence $J$ (not necessarily admissible), recall the related notation, including the heights $l_k^J$ in (\ref{deflev}), the sequence of roots $\Delta$, and its factorization illustrated in \eqref{eqn:delta_factorization}. 

\begin{lemma}{\rm \cite{LenartHHL}[Proposition 3.6]}
	\label{lemma:weight}
	Let $J \subseteq [m]$, and $\sigma=\sfill(J)$. Then we have 
	$\mu(J)=\ct(\sigma)$.
\end{lemma}
\begin{corollary}
	\label{corollary:linfinity}
	Let $J \subseteq [m]$, $\sigma=\fill(J)$, and
	$\alpha \in \Phi$. 
	Then $\sgn(\alpha)l_{\alpha}^{\infty} = \inner{\ct(\sigma)}{\alpha^{\vee}}$.
\end{corollary}

The height counting lemma can be viewed as an extension of Corollary \ref{corollary:linfinity}. 

\begin{lemma}{\rm \cite{LenartHHL}[Proposition 4.1]}
	\label{lemma:height_counting}
Let $J \subseteq [m]$, and $\sigma=\fill(J)$.
For a fixed $k$, let $\gamma_k=(c,d)$ be a root in $\Delta^{q+1}$. 
We have 
	\[
	\sgn(\gamma_k)\,\l{k} = \langle \ct(\sigma[q]),\gamma_k^{\vee}\rangle
		= N_c(\sigma[q]) - N_d(\sigma[q]). 
	\] 
\end{lemma}
	We now introduce notation to be used for the remainder of this section.
	Let $p \in \{1, \dots, n-1 \}$. 
	Let $J$ be an admissible sequence and let $\sigma=\sfill(J)=C_1\ldots C_{\lambda_1}$.
	For $i=1,\ldots,\lambda_1$, let $a_i:=\inner{\ct(C_i)}{\alpha_p^{\vee}}$,  
	and note that $a_i\in \left\{1,-1,0\right\}$; here $a_i =1$ (resp. $a_i=-1$) 
	corresponds to $C_i$ containing $p$ but not $p+1$ (resp. $p+1$ but not $p$), while 
	$a_i = 0$ corresponds to $C_i$ 
	containing both $p$ and $p+1$, or neither of them. 
	
	The sequence $a_i$ corresponds to the $p$-signature from Section 
	\ref{subsection:KR-crystals}, as we now explain. 
	For $j=0,\ldots,\lambda_1$, let $h_j:= \inner{\ct(\sigma[j])}{\alpha_p^{\vee}}=\sum_{i=0}^j a_i$, 
	where we set 
	$a_0=h_0:=0$. It is useful to think of this sequence as a piecewise linear function, by analogy with the function $g_{\alpha_p}$ used to define the crystal operator $f_p$ in the quantum alcove model.	Let $M'$ be the maximum of $h_j$, 
	and let $m'$ be minimal with the property 
	$h_{m'}=M'$.
	We clearly have $M' \geq 0$. If $M'>0$   
	then $a_{m'}=1$, and $m'$ is the position of the rightmost $+$ in the $p$-signature of $\sigma$. 
	It follows that $f_p$ changes the $p$ in column $m'$ of $\sigma$ to $p+1$.
	%
	%
	%
	%
	The previous observations hold if we replace 
	$\alpha_p$, $f_p$, and the $p$-signature with $\alpha_0$, $f_0$, and the $0$-signature, while at the same time we replace the entries $p$ and $p+1$ in a filling with $n$ and $1$,  respectively. Therefore, from now on we assume that the index $p$ is in $\{ 0,1, \ldots, n-1\}$.

	\begin{example}
		We continue with Example \ref{example:filling_map}. 
		Let $\sigma=\sfill(J)=\tableau{ 2 & 1 & 2 & 3 \\ 3 & 2 & 3}$, then 
		$f_2(\sigma)=\tableau{2 & 1 & 2 & 3 \\ 3 & 3 & 3}$.
		Let $p=2$ and refer to
		Figure \ref{fig:exampleA}.
		\begin{figure}[h]
            \centering
            \includegraphics[scale=.5]{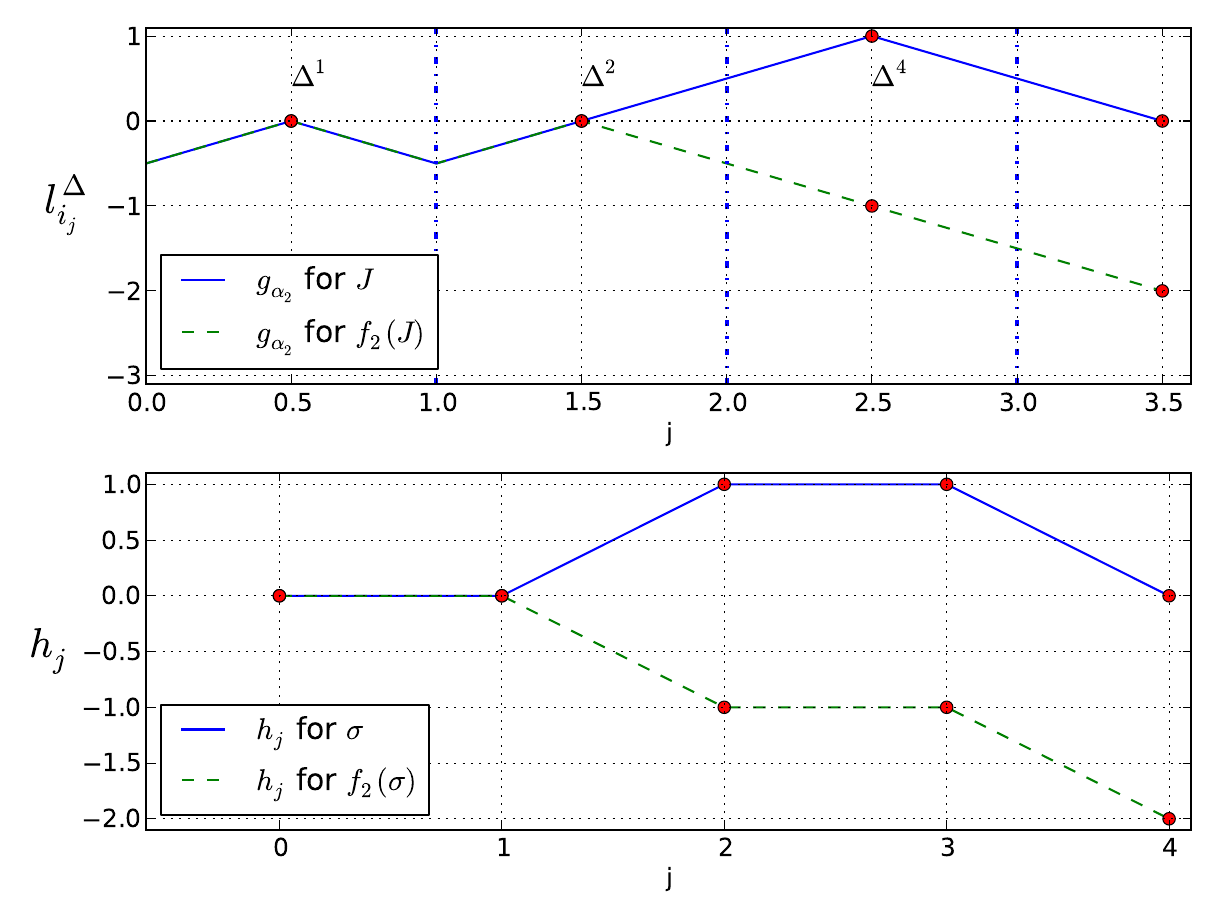}
			\caption{}
			\label{fig:exampleA}
		\end{figure}
		From the graph $g_{\alpha_2}$ for $J$, we can see that $M=1$. We 
		note that $m=7$, with $\gamma_7 \in \Delta^4$, and $k=4$ with 
		$\gamma_k \in \Delta^{2}$. 
		So $f_2(J)=(J \backslash \{ 7 \}) \cup \{ 4 \} = \{1,2,3,4,5 \}$,
		and 
		\[\Gamma(f_2(J))=
		  \{ \underline{(2,3)},\underline{(1,2)} \,|\, \underline{(3,1)}, 
		  \underline{(2,3)} \,|\, \underline{(1,2)},(3,1)\,|\,(3,2),(3,1 \}, \]
		  where we underlined roots in positions $f_2(J)$.
	From the graph corresponding to $h_j$ for $J$, we can see that $M'=1$ and $m'=2$. 

	\end{example}
	\begin{lemma}
		\label{lemma:chain_filling}
		If $\alpha_p=\gamma_k  \in \Delta ^q$ with $k \not \in J$ then $a_q=1$.
	\end{lemma}
		\begin{proof}
			Recall that 
			$\alpha_p=\gamma_k=w(\beta_k)$ for the corresponding $w$ defined in \eqref{defw}, and let $\beta_k = (b,c)$.
			The result follows from the claim that 
			$w(b)=\pi_q(b)$ and $w(c)=\pi_q(c)$ 
			(cf. Definition \ref{definition:fill}),
			which is a consequence of the structure of $\Gamma^q$
			(cf. \eqref{eqn:lambdachainA}), as we now explain.
			The only reflections in $\Gamma^q$ to the right of $\beta_k$ 
			that affect values in positions $b$ or $c$ are
			$(b,c')$ for $c'>c$, and $(b',c)$ for $b'<b$. Applying
			these reflections (on the right) to any $w'$ satisfying 
			$w'(b)=w(b)$ and $w'(c)=w(c)$ 
			 does not give an edge 
			in the quantum Bruhat graph, by the corresponding criterion in Proposition \ref{prop:quantum_bruhat_order_type_A}.
		\end{proof}

	Fix an admissible subset $J$, and recall Proposition $\ref{prop:rootF}$, including the notation therein. In particular, $M$ is the maximum
	of $g_{\alpha_p}$. Moreover, if $M> \delta_{p,0}$ we defined $k\not\in J$ and $m\in J\cup\{\infty\}$ with $\gamma_k=\alpha_p$, $\sgn(\alpha_p)\l{k}=M-1$, 
	and $\gamma_m=\alpha_p$ when  $m \ne \infty$. 
	We will implicitly use the following observation when applying 
	Lemma \ref{lemma:admissible2} in the next two proofs: 
	if $a_i \ne  0$ then $\sgn(a_i) = \sgn(\pi_i^{-1}(\alpha_p))$, where $\pi_i$ is given in Definition \ref{definition:fill}.

		\begin{proposition}
			\label{proposition:max-correspondence}
			We have $M\geq M'$. If $M \geq \delta_{p,0}$ then $M=M'$.
		\end{proposition}
		\begin{proof}
			We first prove that $M\geq M'$. By Corollary \ref{corollary:linfinity} we have 
			$h_{\lambda_1}=\sgn(\alpha)l_{\alpha_p}^{\infty}$, so the case $M'=h_{\lambda_1}$ is clear.  
			The case $M'=0$ is trivial, since $M\geq 0$. Therefore, we can assume that the maximum of the sequence $h_i$ does not occur at its endpoints $i=0$ and $i=\lambda_1$. 
			Then we can find 
			$i<j$ such that $h_i=M'$, $a_i>0$, $a_j<0$, 
			and $a_t=0$ for $t \in (i,j)$. 
			By 
			Lemma \ref{lemma:admissible2} 
			there exists $\gamma_{k'} = \alpha_p \in \Delta^{q}$ with 
			$q \in (i,j]$, and by Lemma \ref{lemma:height_counting} we have $\sgn(\alpha_p)\l{k'}=h_{q-1}=h_i=M'$.
				Hence $M\geq M'$.

			By \eqref{eqn:graph_height}, 
			Lemma \ref{lemma:height_counting},  and 
			Corollary \ref{corollary:linfinity}, all the values of $g_{\alpha_p}$ at points in ${\mathbb Z}+\frac{1}{2}$ are among the values $h_i$. If $M\geq \delta_{p,0}$ then,  
			by Propositions \ref{prop:main1} and \ref{prop:main1_theta}, the maximum $M$ is attained at a point in ${\mathbb Z}+\frac{1}{2}$. It follows that $M \leq M'$, 
			hence $M=M'$.
			\end{proof}

			The previous proposition states that $M=M'$ 
			except in a few corner cases that occur when $p=0$. 
			We will sometimes use one symbol in favor of the other 
			to allude to the corresponding graph.

		\begin{proposition}
			\label{proposition:root_matching}
			Assume that $M>\delta_{p,0}$, so $M=M'$ (by Proposition {\rm \ref{proposition:max-correspondence}}) and $f_p(J)\ne \Bzero$. 
			Then $\gamma_k\in \Delta^{m'}$. If $m \ne \infty$, 
			so $\gamma_m \in \Delta^{m''}$ for some $m''\ge m'$, then $a_i=0$ for $i \in (m',m'')$; if $m= \infty$ then  $a_i=0$ for $i>m'$.
		\end{proposition}

		\begin{proof}
			Assuming  $\gamma_k \in \Delta^j$, 
			by Proposition \ref{prop:rootF} (2), Lemma \ref{lemma:chain_filling}, and Lemma \ref{lemma:height_counting}, we have $a_j=1$ and $\sgn(\alpha_p)\l{k}=M-1=h_{j-1}$. It follows that  
			$h_j=M=M'$. By the definition of $m'$ and the fact that $M'>0$, we have $1\le m'\le j$  and $a_{m'}=1$. 
			By way of contradiction suppose that $m'<j$. 
			It follows that the set
			$\left\{ i \in (m',j] \,|\, a_i \ne 0 \right\}$ is not empty, so let $t$ be its minimum. 
			We have $a_t=-1$ and $t<j$, because $a_t=1$ would imply $h_t>h_{m'}=M'$. 
			We can now apply 
			Lemma \ref{lemma:admissible2} to show that there exists $\gamma_{k'}=\alpha_p\in \Delta^q$ with $q\in(m',t]$ and $k'\in J$. 
			By Lemma \ref{lemma:height_counting}, we have $\sgn(\alpha_p)\l{k'}=h_{q-1}=h_{m'}=M'=M$. Thus, since $k'<k<m$, the
			 minimality of $m$ is contradicted.
			We conclude that $j=m'$.
			If $m\ne \infty$, we can use a similar proof to conclude that
			the set 
			$\left\{i \in (m',m'') \,|\, a_i \ne 0 \right\}$ is empty.
			The case $m=\infty$ is done similarly.
		\end{proof}

	\begin{proof}[Proof of Theorem {\rm \ref{theorem:crystal_isomorphism}}]
	We continue to use the notation from the above setup. Recall that $b=\sfill(J)$.
	The statement that there is a dual Demazure arrow $b \to f_p(b) $ if and only if 
	$f_p(J)\ne \Bzero$ follows from 
    Proposition \ref{proposition:max-correspondence}; indeed, it is clear that $\varphi_p(b)=M'$, whereas $\varphi_p(J)=M-\delta_{p,0}$ in the quantum alcove model if $M\ge\delta_{p,0}$, by Theorem \ref{theorem:admissible} (2).

	We next show that $f_p(b) = \sfill(f_p(J))$, when $f_p(J)\ne \Bzero$. 
	Since $f_p(b) \ne \Bzero$, we have $M'>0$, and $f_p$ changes the 
	$p$ in column $m'$ to  $p+1$ ($f_0$ changes $n$ to $1$ and sorts the column).
	Now let us turn to $f_p(J)$, where we write the admissible subset $J$ as $\{j_1 < \dots < j_s \}$. Let  
	$w_i:=r_{j_1} \dots r_{j_i}$  be the corresponding sequence of permutations. 
       (Recall that the filling $\fill(J)$  is constructed from 
	a subsequence of $w_i$, see Definition \ref{definition:fill}.) 
	We assume  
	$m\ne \infty$, as the case $m=\infty$ is proved similarly.
	There exist $a<b$ such that 
	\[ j_a < k < j_{a+1}  < \cdots < j_b = m < j_{b+1}\,; \]
	if $a=0$ or $b=s$, then the corresponding indices $j_a$, 
	respectively $j_{b+1}$ are missing.
	The sequence of permutations associated to $f_p(J)$ is 
	\[1,\,w_1,\, \ldots,\, w_a,\, s_pw_a,\, s_pw_{a+1},\, \ldots ,\, s_pw_{b-1}=w_b,\, w_{b+1},\, \ldots,\, w_s\]
	(see \eqref{newadmissibleb}).
	By the first part of Proposition \ref{proposition:root_matching} and by using the notation therein, we conclude that 
	$\sfill(f_p(J))$ is obtained from $\sfill(J)$ by interchanging $p$ and $p+1$
	in columns $i$ for $i \in [m',m'')$ (interchange $n$ with $1$ if $p=0$).
	By the second part of Proposition \ref{proposition:root_matching}, this amounts to changing
	the $p$ in column $m'$ to $p+1$ ($n$ to $1$ if $p=0$).
	\end{proof}

\subsection{Type \texorpdfstring{$C$}{C}}
We start with the basic facts about the root system of type $C_n$. We can 
identify the space $\hh^*_{\R}$ with $V:= \R^n$, the coordinate vectors being 
$\varepsilon_1, \dots , \varepsilon_n$. The root system is 
$\Phi = \left\{ \pm \varepsilon_i \pm \varepsilon_j \, :\,  1 \leq i < j \leq n \right\}
\cup
\left\{  \pm 2 \varepsilon_i \, : \, 1 \leq i \leq n \right\}$.
The simple roots are $\alpha_i = \varepsilon_i - \varepsilon_{i+1}$, for 
$i= 1, \dots , n-1, $ and $\alpha_n=2\varepsilon_n$. The highest root 
$\widetilde{\alpha}=2\varepsilon_1$. We let $\alpha_0=\theta=-2\varepsilon_1$. The weight lattice is $\Lambda = \Z^n$.
The fundamental weights are $\omega_i = \varepsilon_1 + \dots + \varepsilon_i$, for 
$i=1, \dots , n$. A dominant weight $\lambda = \lambda_1\varepsilon_1 + \dots + \lambda_n\varepsilon_n$ is
identified with the partition $(\lambda_1 \geq \lambda_2 \geq \dots \geq \lambda_{n-1}\geq \lambda_n \geq 0)$
of length at most $n$. Note that $\rho = (n, n-1, \dots , 1)$.
Like in type $A$, 
writing the dominant weight $\lambda$ as a sum of fundamental weights
corresponds to considering the Young diagram of $\lambda$ as a concatenation of columns.
We fix a dominant weight $\lambda$ throughout this section.

The Weyl group $W$ is the group of signed permutations $B_n$, which acts on $V$ by permuting
the coordinates and changing their signs. A signed permutation is a bijection $w$ from 
$[\overline{n} ]:= 
\{ 1 < 2 < \dots < n < \overline{n} < \overline{n-1} < \dots < \overline{1} \}$ to 
$[\overline{n}]$ satisfying $w(\overline{\imath}) = \overline{w(i)}$.
Here $\overline{\imath}$ is viewed as $-i$, so $\overline{\overline{\imath}}= i$,
$|\overline{\imath}| = i$, and $\sign(\overline{\imath})=-1$.
We use both the window notation $w=w_1 \dots w_n$ and the full one-line notation
$w=w(1)\dots w(n)w(\overline{n})\dots w(\overline{1})$ for signed permutations.
For simplicity, given $1 \leq i < j \leq n$, we denote by $(i,j)$ the root 
$\varepsilon_i - \varepsilon_j$ and the corresponding reflection, which is identified
with the composition of transpositions $t_{ij}t_{\overline{\jmath \imath}}$.
Similarly, we denote by $(i,\overline{\jmath})=(j,\overline{\imath})$, for $1\leq i<j \leq n$, the root 
$\varepsilon_i + \varepsilon_j$ and the corresponding reflection, which is 
identified with the composition of transpositions 
$t_{i\overline{\jmath}} t_{j\overline{\imath}}$.
Finally, we denote by $(i,\overline{\imath})$ the root $2\varepsilon_i$ and the 
corresponding reflection, which is identified with the transposition $t_{i\overline{\imath}}$.

We recall a criterion for the edges of the type $C$ quantum Bruhat graph. We need the circular order $\prec_i$ on $[\overline{n}]$ starting at $i$, which is 
defined in the obvious way, cf. Section \ref{subsection:TypeA}. It is convenient 
to think of this order in terms of the numbers 
$1, \dots, n, \overline{n}, \dots, \overline{1}$ arranged on a circle clockwise.
We make the same convention as in Section \ref{subsection:TypeA} that, whenever we write 
$a \prec b \prec c \prec \dots,$  we refer to the circular order $\prec = \prec_a$.

\begin{proposition}{\rm \cite{Lenart}} 
  ~
	\begin{enumerate}
		\item[{\rm (1)}] 
			 Given $1 \leq i < j \leq n$, we have an edge
			$w \stackrel{(i,j)}{\longrightarrow} w(i,j)$ if and only if
			there is no $k$ such that 
			$i < k < j$ and $w(i) \prec w(k) \prec w(j)$.
		\item[{\rm (2)}] 
			 Given $1 \leq i < j \leq n$, we have an edge 
			$w\stackrel{(i,\overline{\jmath})}{\longrightarrow}$ if and
			only if $w(i) < w(\overline{\jmath})$,  
			$\sign(w(i)) = \sign(w(\overline{\jmath}))$, 
			and there is no $k$ such that $i<k<\overline{\jmath}$ and 
			$w(i) < w(k) < w(\overline{\jmath})$.
		\item[{\rm (3)}] 
			 Given  $1 \leq i \leq n$, we have an edge 
			$w \stackrel{(i,\overline{\imath})}{\longrightarrow} 
			w(i,\overline{\imath})$ if and only if there is no $k$ such
			that $i < k < \overline{\imath}$ (or, equivalently, $i<k\leq n$) and 
			$w(i) \prec w(k) \prec w(\overline{\imath})$.
	\end{enumerate}
	\label{prop:quantum_bruhat_order_type_C}
\end{proposition}

We now consider the specialization of the quantum alcove model to type $C$. 
For any $k=1, \ldots , n$, we have the following $\omega_k$-chain, from $A_{\circ}$ to 
$A_{-\omega_k}$, denoted by 
	$\Gamma(k)$ \cite{LenartHHL_height_counting}:
	\begin{align}
		\Gamma(k):=\, & \Gamma_l(k)\Gamma_r(k) \text{ where }\nonumber \\ 
		\Gamma_l(k):=\, & \Gamma_{kk} \dots \Gamma_{k1},\;\;\;
		\Gamma_r(k):=\, \Gamma_k \dots \Gamma_2,\, \label{eqn:lambdachainC} \\ 
		\Gamma_i:=\, & 
		\left( 
		(i,\overline{i-1}), (i, \overline{i-2}), \dots , (i, \overline{1}) 
		\right),\,\nonumber \\ 
		\begin{matrix}
			\Gamma_{ki} :=\, \\\mbox{} \\\mbox{} \\\mbox{}   
		\end{matrix}
		&
		\begin{matrix*}[l]
			( (i,k+1),&    (i,k+2),&     \ldots,&  (i,n),  \nonumber  \\
			\phantom{(} (i,\overline{\imath}), & & & \nonumber \\ 
			\phantom{(} 
		 (i,\overline{n}),& (i, \overline{n-1}),& \dots,& (i, \overline{k+1}),\nonumber \\ 
		 \phantom{(} 
		 (i,\overline{i-1}),&  (i,\overline{i-2}),& \ldots,& (i,\overline{1}))\,.\nonumber  
		\end{matrix*}
	\end{align}

Fix a dominant weight/partition $\lambda$ for the remainder of this section. We construct a $\lambda$-chain 
$\Gamma=(\beta_1, \beta_2, \dots, \beta_m)$ as a concatenation 
$\Gamma:= \Gamma^1 \dots \Gamma^{\lambda_1}$, where $\Gamma^j= \Gamma(\lambda'_j)$;
we also let $\Gamma_l^j:=\Gamma_l(\lambda'_j)$ and $\Gamma_r^j:=\Gamma_r(\lambda'_j)$.
Like in type $A$, given a set $J=\left\{ j_1 < \dots < j_s \right\}$ 
of folding positions in $\Gamma$, not necessarily admissible, we  
let $T$ be the corresponding list of 
roots of $\Gamma$.
We factor $\Gamma$ as $\Gamma=\widetilde{\Gamma}^1\dots \widetilde{\Gamma}^{2\lambda_1}$, where
$\widetilde{\Gamma}^{2i-1}=\Gamma^i_l$ and $\widetilde{\Gamma}^{2i}=\Gamma^i_r$, for
$1\leq i \leq \lambda_1$.
This factorization of $\Gamma$ induces a factorization of $T$ as 
$T^1T^2 \dots T^{2\lambda_1}$, and of $\Delta=\Gamma(J)$ as $\Delta=\Delta^1 \dots \Delta^{2\lambda_1}$. 
Like in type $A$, we use the notation $\gamma_k\in\Delta^q$ to indicate that the $k$th root in $\Delta$ falls in the segment $\Delta^q$. 
We denote by $T^1T^2 \dots T^j$ the permutation obtained by
composing the type $C$ transpositions in $T^1, \dots , T^j$ left to right.
For $w \in W$ written in the window notation as $w=w_1w_2 \dots w_n$, let $w[i,j]=w_i \dots w_j$.

We now recall from \cite{Lenart} the construction of the correspondence between the type $C$ quantum alcove model and model based on diagram fillings. 

	  \begin{definition}
		  \label{definition:fillTypeC}
	  Let $\pi_{j}=\pi_{j}(T):=T^1 \dots T^j$. We define the \emph{filling map},
	  which associates with each $J\subseteq[m]$ a filling of the Young diagram $2\lambda$, by 
	  \begin{equation}
		  \label{eqn:filling_mapTypeC}
		  \fill(J)=\fill(T):=C_{1}\dots C_{2\lambda_1}\,,\;
		  \mbox{ where } C_{i}:=\pi_i[1,\lambda'_{\lceil \frac{i}{2}\rceil}].
	  \end{equation}
	  We define the \emph{sorted filling map} $\sfill(J)$ by sorting ascendingly the 
	columns of the filling $\fill(J)$.
	  \end{definition}
	  
	  For an example we refer to \cite{Lenart}[Examples 5.3 and 5.5].



Recall from \eqref{btensl} the definition of $B^{\otimes \lambda}$, which is now realized with split KN columns, see Section \ref{subsection:KR-crystals}. As such, its arrows are given by $f_p^2$, according to Proposition \ref{typeCcryst} (2).

\begin{theorem}{\rm \cite{Lenart}[Theorem 6.1]}\label{theorem:bijectionTypeC}
	The map $\sfill$ is a bijection between $\A(\Gamma)$ and $\B$.
\end{theorem}
We now state the main result of this section,  cf. Theorem {\rm \ref{theorem:crystal_isomorphism}} in type $A$.
\begin{theorem}
	\label{theorem:crystal_isomorphismTypeC}
	The map $\sfill$ is an affine crystal isomorphism between ${\mathcal A}(\Gamma)$ and the subgraph of
$B^{\otimes \lambda}$ consisting of the dual Demazure arrows.
\end{theorem}

\begin{remarks}\label{nddC} (1) The affine crystal isomorphism in Theorem \ref{theorem:crystal_isomorphismTypeC} is not guaranteed to be unique in general. However, we believe that it coincides with the bijection that we plan to construct (in a type-independent setup) in order to prove Conjecture \ref{mainmainconj}. See Remark \ref{perfectcase} (1).  

(2) In \cite{Lenart} it was shown that the map $\sfill$ translates the height statistic to the charge statistic, which is known to express the energy function in the model based on KN columns, cf. Theorem \ref{theorem:charge}. This should be compared with Theorem \ref{mainconj} (2), where the constant $C$ is $0$ in this case.
\end{remarks}

The proof of Theorem \ref{theorem:bijectionTypeC} is parallel to the proof of Theorem \ref{theorem:crystal_isomorphism}. 
In this case, we use the height counting lemma in
type $C_n$, namely Lemma \ref{lemma:height_countingTypeC}. 
As before let $N_i(\sigma)$ denote the number of entries $i$ in a filling $\sigma$.
Let $c_i=c_i(\sigma) := \frac{1}{2} ( N_i(\sigma) - N_{\overline{\imath}}(\sigma)$) and define
the content of a filling $\sigma$ as $\ct(\sigma):=(c_1,c_2,\dots,c_n)$, which is identified with a type $C_n$ weight.
Let $\sigma[q]$ be the filling consisting of the columns $1,2, \dots , q$ of $\sigma$.
Given a $\lambda$-chain and a corresponding sequence $J$ (not necessarily
admissible), recall the related notation, including the heights $l_k^J$ in (\ref{deflev}), the sequence of roots $\Delta$, 
and its factorization.

\begin{lemma}{\rm \cite{LenartHHL_height_counting}[Proposition 4.6 (2)]}
	\label{lemma:weightTypeC}
	Let $J \subseteq [m]$, and $\sigma=\fill(J)$. Then we have 
	$\mu(J)=\ct(\sigma)$.
\end{lemma}
\begin{corollary}
	\label{corollary:linfinityTypeC}
	Let $J \subseteq [m]$, $\sigma=\fill(J)$, and
	$\alpha \in \Phi$. 
	Then $\sgn(\alpha)l_{\alpha}^{\infty} = \inner{\ct(\sigma)}{\alpha^{\vee}}$.
\end{corollary}
\begin{lemma}{\rm \cite{LenartHHL_height_counting}[Proposition 6.1]}
	\label{lemma:height_countingTypeC}
Let $J \subseteq [m]$, and $\sigma=\fill(J)$.
For a fixed $k$, let $\gamma_k$ be a root in $\Delta^{q+1}$. 
We have 
	\[
	\sgn(\gamma_k)\,\l{k} = \langle \ct(\sigma[q]),\gamma_k^{\vee}\rangle.
	\] 
\end{lemma}

%

	We now introduce notation to be used for the remainder of this section.
	Let $p \in \{0,1, \dots, n \}$. 
	Let $J$ be an admissible sequence, and $\sigma=\sfill(J)=C_1\ldots C_{2\lambda_1}$, which is guaranteed to be in $B^{\otimes \lambda}$ by Theorem \ref{theorem:bijectionTypeC}. 
	Let $a_i:=\inner{\ct(C_i)}{\alpha_p^{\vee}}$. We have $a_i\in \left\{-1,-\frac{1}{2},0,\frac{1}{2},1\right\}$,
	for $1\leq p \leq n-1$, and $a_i \in \left\{ -\frac{1}{2},0,\frac{1}{2} \right\}$ for
	$p\in \left\{ 0,n \right\}$.
	
	\begin{remarks}\label{allpcases}
	(1) The value of $a_i$ indicates which entries related to the action of $f_p$ are contained in column $C_i$, as we now explain. Assuming first that $1 \leq p \leq n-1$, the relevant entries are ${\mathcal P}:=\{p,\,p+1,\,\overline{p+1},\,\overline{p}\}$. 
	 If $a_i = 1$ (resp. $a_i=-1$), then $C_i$ contains both $p$ and $\overline{p+1}$ (resp. $p+1$ and $\overline{p}$). If $a_i = \frac{1}{2}$ (resp. $a_i=-\frac{1}{2}$), then $C_i$ contains only one of $p,\,\overline{p+1}$ (resp. only one of $p+1,\,\overline{p}$), while if $a_i=0$ then $C_i$ contains both $p$ and $p+1$, or both $\overline{p+1}$ and $\overline{p}$, or none of these elements.
	 For $p = n$, the relevant entries are $n$ and $\overline{n}$. If $a_i = \frac{1}{2}$ (resp. $a_i=-\frac{1}{2}$), then $C_i$ contains $n$ (resp. $\overline{n}$), while if $a_i=0$  then  $C_i$ contains both $n$ and $\overline{n}$, or none of these elements. The case $p=0$ is similar to $p=n$: just replace $n$ and $\overline{n}$ with $\overline{1}$ and $1$, respectively. 
		
	(2)	The sequence $a_i$ corresponds to the $p$-signature of the filling $\sigma$. To be more precise, associate with this sequence a $(+,-)$-word by replacing $\pm\frac{1}{2}$ with $\pm$ and $\pm 1$ with $\pm\pm$ (the $0$'s are ignored). If $1\le p\le n-1$, this is the same as the $(+,-)$-word associated with $\sigma$ (see Section \ref{subsection:KR-crystals}) after cancelling $-+$ pairs corresponding to the entries $p$ and $p+1$ (or $\overline{p+1}$ and $\overline{p}$) in a column; similarly for $p=n$ and $p=0$. 
	\end{remarks}
		
	Let $h_j:= \inner{\ct(\sigma[j])}{\alpha_p^{\vee}}=\sum_{i=0}^j a_i$, 
	with $a_0=h_0:=0$. Like in type $A$, let $M'\ge 0$ be the maximum of $h_j$, 
	and let $m'$ be minimal with the property 
	$h_{m'}=M'$. If $M'>0$ then $m'$ is the number of the column containing the entry changed by $f_p$. Recall that  we need to apply $f_p$ twice; the way in which this can happen is described below. 

\begin{proposition}\label{actfp2} If $1\le p\le n-1$, then we always have one of the following cases related to the action of $f_p^2$ on the filling $\sigma$. 
\begin{enumerate}
\item[{\rm (i)}] $m'=2i-1$ and $a_{m'}=1$: $p$ and $\overline{p+1}$ in column $m'$ are changed to $p+1$ and $\overline{p}$.
\item[{\rm (ii)}] same as {\rm (i)} with $m'=2i$. 
\item[{\rm (iii)}] $m'=2i$ and $a_{m'}=a_{m'-1}=\frac{1}{2}$: columns $m'$ and $m'-1$ both contain an entry $p$ (or both contain $\overline{p+1}$), and these entries are changed to $p+1$ (resp. $\overline{p}$).
\end{enumerate}
If $p=n$ or $p=0$, then the analogue of case {\rm (iii)} always holds, with $n$ changed to $\overline{n}$, resp. $\overline{1}$ changed to $1$. 
\end{proposition}

\begin{proof} We implicitly use the following observation, which is immediate from the construction of the splitting $(lC,rC)$ of a column $C$ in Definition \ref{definition:KN_doubled_columns}: given $x\in[n]$, the column $lC$ contains $x$ or $\overline{x}$  if and only if $rC$ does. 

We consider only $1\le p\le n-1$, as the proof is simpler for $p=n$ and $p=0$. We first prove the following claim: if $a_{2i-1}=\pm\frac{1}{2}$ (or $a_{2i}=\pm\frac{1}{2}$), then $C_{2i-1}$ and $C_{2i}$ contain a single element in ${\mathcal P}$, the two elements have the same absolute value, and the pair $(a_{2i-1},\,a_{2i})$ can take only the following values: $\left(\frac{1}{2},\,\frac{1}{2}\right)$, $\left(-\frac{1}{2},-\,\frac{1}{2}\right)$, or  $\left(-\frac{1}{2},\,\frac{1}{2}\right)$. 

We consider only the case $a_{2i-1}=\frac{1}{2}$, as the others are completely similar. The assumption implies that column $C_{2i-1}$ contains a single element in $\mathcal P$, namely $p$ or $\overline{p+1}$. In the first case, it is clear that $C_{2i}$ does not contain $p+1$ or $\overline{p+1}$, but it contains either $p$ or $\overline{p}$. It suffices to rule out the occurence of $\overline{p}$. Assuming it, we deduce that the column $C$ whose splitting is $(C_{2i-1},C_{2i})$ contains $z_j>p$ and $\overline{z_j}$, and the corresponding $t_j<z_j$ is $p$ (cf. Definition \ref{definition:KN_doubled_columns}). But $C$ cannot contain $p+1$ or $\overline{p+1}$, so $z_j>p+1$, and the maximality of $t_j$ is contradicted. In the second case, we need to rule out the occurence of $p+1$ in $C_{2i}$. Assuming it, we deduce that $C$ contains $z_j=p+1$ and $\overline{z_j}$, and we have $t_j<p$, so again the maximality of $t_j$ is contradicted. 

Now consider the $(+,-)$-word associated with the sequence $a_j$, see Remark \ref{allpcases} (2). Cancel pairs $-+$ corresponding to the case $(a_{2i-1},\,a_{2i})=\left(-\frac{1}{2},\,\frac{1}{2}\right)$ mentioned above. The above claim implies that the resulting word is a concatenation of pairs $++$ and $--$ which come from $a_j=\pm 1$ and $(a_{2i-1},\,a_{2i})=\left(\pm\frac{1}{2},\,\pm\frac{1}{2}\right)$; recall that for the latter pairs, the claim also gives the corresponding entries in $\mathcal P$. The statement of the proposition now follows.
\end{proof}

	The following is the analogue of Lemma \ref{lemma:chain_filling}.
	\begin{lemma}
		\label{lemma:chain_fillingTypeC}
		If $\alpha_p=\gamma_k  \in \Delta^{2i-1}$ with $k \not \in J$ then we have either
		$a_{2i-1}=\frac{1}{2}$ and $a_{2i}=\frac{1}{2}$, or 
		$a_{2i-1}=1$.
		If $\alpha_p=\gamma_k \in \Delta^{2i}$ with $k \not \in J$ then $a_{2i}=1$.
	\end{lemma}
	
	\begin{proof} We only consider the case corresponding to $\alpha_p\in\Delta^{2i-1}$ and $1\le p\le n-1$, as the others are simpler. We use freely the structure of the chain of roots $\widetilde{\Gamma}^{2i-1}$, see \eqref{eqn:lambdachainC}. Recall that 
			$\alpha_p=\gamma_k=w(\beta_k)$ for the corresponding $w$ defined in \eqref{defw}. We have the following cases:
			\begin{enumerate}
			\item $\beta_k=(b,c)$ with $b\le \lambda_i'$, $c>\lambda_i'$, and $w(b)=p$, $w(c)=p+1$ (or $w(b)=\overline{p+1}$, $w(c)=\overline{p}$);
			\item $\beta_k=(b,\overline{c})$ with $b\le \lambda_i'$, $c>\lambda_i'$, and $w(b)=p$, $w(c)=\overline{p+1}$ (or $w(b)=\overline{p+1}$, $w(c)={p}$);
			\item $\beta_k=(b,\overline{c})$ with $c<b\le \lambda_i'$, and $w(b)=p$, $w(c)=\overline{p+1}$ (or $w(b)=\overline{p+1}$, $w(c)={p}$).
			\end{enumerate}
			
			We will only consider the first case with $w(b)=p$, $w(c)=p+1$, as the others are completely similar. We first claim that $w(b)=\pi_{2i-1}(b)$ (cf. Definition \ref{definition:fill}), i.e., the entry $p$ is not moved by the reflections $s_{\beta_{k'}}$ for $\beta_{k'}$ in $\widetilde{\Gamma}^{2i-1}$ with $k'>k$. These reflections are $(b,c')$ with $c'>c$, $(b,\overline{b})$, $(b,\overline{c'})$ with $c'>\lambda_i'$, and $(b,\overline{c'})$ with $c'<b$. So if the claim failed, the quantum Bruhat graph criterion in Proposition \ref{prop:quantum_bruhat_order_type_C}  would be violated, because the entry $p+1$ is still in position $c$ when these reflections are applied. 
			
			Let us now track the entry $p+1$ as we apply the subsequent reflections not involving position $b$, by freely using the quantum Bruhat graph criterion. If this entry is not moved by any of these reflections, then $\pi_{2i-1}(c)=p+1$, so $a_{2i-1}=\frac{1}{2}$. The first reflection which can move $p+1$ is of the form $(b',\overline{c})$ with $b'<b$, which means that in position $b'$ we will now have the entry $\overline{p+1}$. If this entry is not moved by any of the subsequent reflections $(b',\overline{c'})$ with $\lambda_i'<c'<c$, then it is not moved by any of the remaining reflections either, so $\pi_{2i-1}(b')=\overline{p+1}$, and $a_{2i-1}=1$. Otherwise, we will have the entry $p+1$ in a position $c'$ with $\lambda_i'<c'<c$, and the above reasoning can be applied again (a finite number of times). 
			
			Finally, the fact that if $a_{2i-1}=\frac{1}{2}$ then $a_{2i}=\frac{1}{2}$ was deduced in the proof of Proposition \ref{actfp2}.
			\end{proof}
			
	Recall Proposition $\ref{prop:rootF}$ and the notation therein. $M$ is the maximum
	of $g_{\alpha_p}$, and suppose $M> \delta_{p,0}$; then $\gamma_k=\alpha_p$ with 
	$k \not \in J$, $\sgn(\alpha_p)\l{k}=M-1$, 
	and if $m \ne \infty$ then $\gamma_m=\alpha_p$ with $m \in J$. 
	The following result is the analogue of Proposition \ref{proposition:max-correspondence}, and its proof is identical. Indeed, the following key fact is still true: if $a_i \ne  0$ then $\sgn(a_i) = \sgn(\pi_i^{-1}(\alpha_p))$, where $\pi_i$ is given in Definition \ref{definition:fillTypeC} (simply note that $a_i=\langle\pi_i(\omega_k),\alpha_p^\vee\rangle=\langle\omega_k,\pi_i^{-1}(\alpha_p^\vee)\rangle$, where $k=\lambda_{\left\lceil \frac{i}{2}\right\rceil}'$). This will be needed in the proof of Proposition \ref{proposition:root_matchingTypeC} as well.
	
		\begin{proposition}
			\label{proposition:max-correspondenceTypeC}
			We have $M\geq M'$. If $M \geq \delta_{p,0}$, then $M=M'$.
		\end{proposition}
		
		The following result is the analogue of Proposition \ref{proposition:root_matching}.

		\begin{proposition}
			\label{proposition:root_matchingTypeC}
			Assume that $M>\delta_{p,0}$, so $M=M'$ (by Proposition {\rm \ref{proposition:max-correspondenceTypeC}}) and $f_p(J)\ne \Bzero$. 
			If $a_{m'}=1$ then $\gamma_k\in \Delta^{m'}$, otherwise
			$\gamma_k \in \Delta^{m'-1}$. If $m \ne \infty$, 
			so $\gamma_m \in \Delta^{m''}$, then $a_i=0$ for $i \in (m',m'')$; if $m= \infty$, then  $a_i=0$ for $i>m'$.
		\end{proposition}
	
	\begin{proof}
			Assuming  $\gamma_k \in \Delta^j$, 
			by Proposition \ref{prop:rootF} (2), Lemma \ref{lemma:chain_fillingTypeC}, and Lemma \ref{lemma:height_countingTypeC}, we have  $\sgn(\alpha_p)\l{k}=M-1=h_{j-1}$, and either $a_j=1$ or $a_j=a_{j+1}=\frac{1}{2}$. In the first case, the rest of the proof is essentially identical to that of Proposition \ref{proposition:root_matching}; in particular, we show that $j=m'$. 
		In the second case, we have $j=2i-1$, so it follows that  
			$h_{2i}=M=M'$. Once again, essentially the same proof as that of Proposition \ref{proposition:root_matching} applies; in particular, we show that $2i=m'$. Note that in both situations we implicitly used the cases in Proposition \ref{actfp2}. 
			\end{proof}

	\begin{proof}[Proof of Theorem {\rm \ref{theorem:crystal_isomorphismTypeC}}]
	The proof is similar to that of Theorem  \ref{theorem:crystal_isomorphism}, so we only point out the extra complexity in type $C$. This has to do with showing that, if $b=\sfill(J)$ and $f_p(J)\ne \Bzero$, then
$f_p^2(b) = \sfill(f_p(J))$. We continue to use the notation from the above setup, and we consider only the case $1\le p\le n-1$, as the others are simpler. 

Since $f_p(b) \ne \Bzero$, we have $M'>0$, and $f_p^2$ acts on $b$ in one of the ways indicated in Proposition \ref{actfp2}, cases (i)-(iii). 
		Now let us turn to $f_p(J)$, and use the same setup as in the proof of Theorem \ref{theorem:crystal_isomorphism}, to which we refer. 
	By the first part of Proposition \ref{proposition:root_matchingTypeC}, we conclude that 
	$\sfill(f_p(J))$ is obtained from $\sfill(J)$ by applying $s_p$ to columns $i$ for  $i \in [m',m'')$ in cases (i)-(ii), resp. $i \in [m'-1,m'')$ in case (iii).
	By the second part of Proposition \ref{proposition:root_matchingTypeC}, this amounts to applying $s_p$ to column $m'$, resp. to columns $m'-1$ and $m'$. 
	By Proposition \ref{actfp2}, this is the same as the action of $f_p^2$ on $\sfill(J)$, which concludes the proof. Note that in the above reasoning we implicitly used Remark \ref{allpcases} (1).
	\end{proof}

\bibliographystyle{alpha}

\newcommand{\etalchar}[1]{$^{#1}$}

\end{document}